\documentclass[preprint,12pt,authoryear]{elsarticle}

\usepackage{amssymb}

\usepackage{amssymb,amsmath,bm,mathrsfs,amsfonts,graphicx,multirow,stmaryrd,color}
\usepackage{times}
\usepackage{natbib,dsfont}
\usepackage{amsthm}
\usepackage{booktabs}
\usepackage{epstopdf}
\usepackage{graphics}
\allowdisplaybreaks[1]

\newcommand{\bA}{{\bf A}}
\newcommand{\bB}{{\bf B}}
\newcommand{\bE}{{\bf E}}

\newcommand{\bx}{{\bf x}}

\newcommand{\um}{\underline{m}}

\newcommand{\bo}{{\bf 0}}
\newcommand{\bF}{{\bf F}}

\newcommand{\bI}{{\bf I}}

\newcommand{\bY}{{\bf Y}}
\newcommand{\by}{{\bf y}}
\newcommand{\bX}{{\bf X}}

\newcommand{\bZ}{{\bf Z}}

\newcommand{\bW}{{\bf W}}

\newtheorem{theorem}{Theorem}[section]

\newtheorem{remark}{Remark}[section]

\newtheorem{corollary}{Corollary}[section]
\newtheorem{proposition}{Proposition}[section]

\newcommand{\bqa}{\begin{eqnarray}}
\newcommand{\eqa}{\end{eqnarray}}
\newcommand{\bqn}{\begin{eqnarray*}}
\newcommand{\eqn}{\end{eqnarray*}}
\newcommand{\be}{\begin{equation}}
\newcommand{\ee}{\end{equation}}
\newcommand{\non}{\nonumber\\}

\newcommand{\rE}{{\rm E}}
\newcommand{\tr}{{\rm tr}}

\newcommand{\md}{\mbox{d}}

\numberwithin{equation}{section}
\theoremstyle{plain}

\begin{document}

\begin{frontmatter}

\title{Generalized  Four Moment  Theorem with  an application to the  
CLT  for the spiked eigenvalues of high-dimensional general Fisher-matrices.}



\author[els]{Dandan Jiang\corref{cor1}}
\cortext[cor1]{Corresponding author: Dandan Jiang}
\ead{jiangdandan@jlu.edu.cn}
\fntext[myfootnote]{Supported by Project 11471140 from NSFC.}
\author{Zhiqiang Hou}
\ead{houzq399@nenu.edu.cn}
\author{Zhidong Bai}
\ead{baizd@nenu.edu.cn}

\address{School of Mathematics and Statistics, \\
Xi'an Jiaotong University \\
No.28, Xianning West Road,\\
Xi'an {\rm 710049}, China.}

 \address{KLASMOE and School of Mathematics and Statistics, \\
Northeast Normal University, \\
No. 5268 People's Street,\\
Changchun {\rm 130024}, China.}

\begin{abstract}
 The universality  for the local spiked eigenvalues  
  is a powerful tool to deal with the problems of the asymptotic law for the bulks  of spiked eigenvalues of high-dimensional generalized Fisher matrices. In this paper, we focus on a more generalized spiked Fisher matrix, where 
  $\Sigma_1\Sigma_2^{-1}$
is free of the  restriction of diagonal  independence,
and both of  the spiked eigenvalues and the population 4th moments  are not necessary required to be bounded. 
 By reducing the matching  four moments  constraint to a tail probability,  we propose a  {\it Generalized Four Moment Theorem} (G4MT) for the bulks  of spiked eigenvalues of high-dimensional generalized Fisher matrices, which 
  shows that the limiting distribution of  the spiked  eigenvalues of a generalized spiked Fisher matrix is independent of the actual distributions of the samples provided to satisfy the our relaxed assumptions. 
  Furthermore, as an illustration, we also  apply the G4MT to  the Central Limit Theorem  for the spiked  eigenvalues of generalized spiked Fisher matrix, which removes the strict condition of the diagonal block independence given in \cite{WangYao2017} and extends their result 
 to a wider usage without  the requirements of the bounded 4th moments and the diagonal block independent structure, meeting the actual cases better.
\end{abstract}


\end{frontmatter}

\section{Introduction} \label{Int}

We study the universality  for the bulks  of spiked eigenvalues of high dimensional generalized Fisher matrices, which plays an important role in many fields  of modern science, such as 
wireless communications,  gene expression and so on.
To formulate the problem in a general form,  let $\Sigma_1$ and $\Sigma_2$ be any covariance matrices from arbitrary two $p$-dimensional populations. Let $S_1$ and $S_2$ denote the corresponding 
sample covariance matrices with sample sizes $n_1$ and $n_2$. 
If  the covariance matrix of the observed vector satisfies $\Sigma_2=\Sigma_1+\Delta$,   where
$\Delta$ is 
a $p \times p$ matrix of  finite rank $M$, then the matrix $\bF=S_1S_2^{-1}$ is so-called a spiked Fisher matrix.
In the present paper, the universality  for the bulks  of spiked eigenvalues of the generalized Fisher matrix $\bF$ is established under the more general assumptions detailed as below: First, 
  the  spectrum of $\bF$ is formed as 
\begin{equation}
\beta_{p,1}, \cdots,  \beta_{p,j},\cdots,\beta_{p,p}\label{array}
\end{equation}
in descending order
and let $\beta_{p,j_k+1}= \cdots= \beta_{p, j_k+m_k}\triangleq \alpha_k$ with $j_k's$ being  arbitrary  ranks in the array 
(\ref{array}), then  the spiked eigenvalues  $\alpha_1, \cdots, \alpha_K$ with multiplicity $m_k, k=1,\cdots,K$  are lined arbitrarily in groups among all the eigenvalues, satisfying $m_1+\cdots+m_K=M$, a fixed integer.
In addition,  the spiked eigenvalues are allowed  to be infinity.  Under these general assumptions, the matrix $\bF$ is called a generalized spiked Fisher matrix.

Our main goal is to construct a Generalized Four Moment Theorem (G4MT)  to shows the universality of the asymptotic law for  the bulks  of spiked eigenvalues of generalized Fisher matrix $\bF$.
On the basis of this preliminary result, the central limit theorem (CLT) for the spiked eigenvalues of generalized Fisher matrix with relaxed  assumptions is also provided as an application.  
%

 Our work  mainly arises from two aspects of  impressive related works: universality and spiked model. There exists a number of literatures on both  topics. On the one hand,
  the study of universality  for the local spectral statistics of random matrices starts from \cite{Wigner1958}, \cite{Dyson1970}, and \cite{Mehta1967}, which provide a new and simplified technique to prove one result suitable for Non-Gaussian case, that is, sufficient to show that the same conclusion hold for the Gaussian case if the universality is true.
  Further catalytic works are also introduced in \cite{Soshnikov1999}, \cite{BP2005}, \cite{Erdos2010a,Erdos2010b}.
   A more general and recent work of universality is 
    \cite{TaoVu2015}, which  proves the universality for the local spectral statistics of the Wigner matrix by the {\it Four Moment Theorem}.  
 This theorem assumes the corresponding equality of the moments up to the 4th order  between the entries from  the complex  standardized Gaussian ensemble and the ones from  the complex  standardized Non-Gaussian ensemble. However, they also  conjectured that
   the number of matching moments  may be  reduced in their theorem. 
   Inspired by these ideas, we have extended it to a G4MT (Generalized Four Moment Theorem) with the relaxed 4th moment constraint in \cite{JiangBai2018}, 
   which shows the universality of the asymptotic law for the local spectral statistics of generalized spiked covariance matrices. On the basis of the previous works, we further develop
   the G4MT for the bulks of the spiked eigenvalues of the high-dimensional
  spiked Fisher matrices in the present paper, which will have a wider range of usage
  in statistical analysis.

   On the other hand, our work relies on  spiked population model,
   a popular theoretical tool in statistical analysis, which has a close relationship with  principal component analysis (PCA) and factor analysis (FA). It was first put forward 
   in  \cite{Johnstone2001} in the setting of high dimensionality $p$ compared to the sample size $n$. It is assumed that the population covariance matrix has a structure of a finite-rank perturbation of identity matrix in their work. For a small finite-rank $M$, the empirical spectral distribution of the corresponding sample covariance matrix still follows 
   the standard Mar\v{c}enko-Pastur law. But the limiting behavior of bulks of spiked eigenvalues are different from the ones of a covariance without spikes. Then,  there has been a lot of work  focused  on the research of 
 the asymptotic properties of the spiked eigenvalues of high-dimensional 
 covariance matrices, including  
  \cite{Baik2005}, \cite{BaikSilverstein2006}, \cite{Paul2007}
and \cite{BaiYao2008}. For a further step,  
  \cite{BaiYao2012} expand the structure of covariance to  a more general spiked covariance matrix with  the block independence and finite 4th moment condition. 
 Some related references are also  devoted to the investigations on PCA or FA, which can be seen as another way of   understanding the spiked model.
      For examples,   \cite{BaiNg2002, HoyleRattray2004,  Nadler2008, 
       Onatski2009, Onatski2012, JungMarron2009, Shen2013, BerthetRigollet2013, Birnbaum2013}  and so on. 
\cite{FanWang2015} provided the asymptotic distributions of 
the sample eigenvalues  and eigenvectors of the  spiked covariance matrix,
 which is a natural extension of \cite{Paul2007}.
 \cite{CaiHanPan2017} constructed  the limiting normal distribution for the spiked  eigenvalues of the sample covariance matrices, which is depended on the population eigenvectors and finite 4th moment.
  A closer work is our recent  results   in 
   \cite{JiangBai2018},  
  which
   extended  the study on the spiked eigenvalues of a  covariance matrix  to a more general case 
   by their G4MT for the general covariance matrix.


    
    However, the above works are  more focused on
     the high-dimensional 
    spiked covariance matrices, and few of them are  referred to spiked Fisher matrices. As well known, the Fisher matrices 
    have an important position  in
multivariate statistical analysis, because many hypothesis testing problems
can be involved with a function of  the eigenvalues of Fisher matrices. To enumerate,  tests on the equality of means or two population covariance matrices, the likelihood ratio criterion for testing regression coefficients in linear regression, the canonical correlation analysis and so on. 
A closest result of local spiked eigenvalues of  high-dimensional Fisher matrices is provided in \cite{WangYao2017}. They established CLT for the extreme eigenvalues  of high-dimensional spiked 
Fisher matrices under the simplified assumption that  $\Sigma_1\Sigma_2^{-1}$ is a rank $M$ perturbation of identity matrix with diagonal independence and bounded 4th moment. 
Therefore,
inspired by   these works,
we consider the limiting behavior of bulks of spiked eigenvalues of high-dimensional spiked Fisher matrices  in a generalized case that  $\Sigma_1\Sigma_2^{-1}$
is free of the  restriction of diagonal  independence,
and both of  the spiked eigenvalues and the population 4th moments  are not necessary required to be bounded. 
Under these relaxing constraints,  
a G4MT is established for the spiked eigenvalues of generalized Fisher matrix $\bF$  with a relaxed  4th-moment constraint, which shows the universality of the asymptotic law for the local eigenvalues of generalized spiked Fisher matrices. 
Then, by applying  the G4MT, the CLT for the spiked  eigenvalues of generalized spiked Fisher matrices is also proposed 
   under some relaxed   assumptions, including  the arbitrary form of the definite matrix $\Sigma_1\Sigma_2^{-1}$
without  diagonal  independence,
the spiked eigenvalues and the population 4th moments not necessarily bounded, free of population distributions. 

Compared  with the existing research, our main contribution mainly manifests in the following several aspects:
One one side,  we establish a G4MT to prove the the universality of the asymptotic law  for  the spiked eigenvalues of generalized Fisher matrices by replacing the condition of matching the 4th moment 
with a tail probability condition, which is a regular and necessary condition in the weak convergence of the largest eigenvalue. 
Thus,  it weakens the condition of  matching moments up to the 3th order,  and even up to the second moments
for the  symmetric populations. Moreover, it avoids  the  rigorous $C_0$ condition with  uniform exponential decay  and the partial derivative operations  of the whole  large dimensional random matrices in \cite{TaoVu2015}, instead it only need to 
study the universality  of a limiting law for the eigenvalues of a low-dimensional $M \times M$ matrix.
In addition, by the G4MT,
our conclusion can be free of the population distribution and the constraint of bounded 4th moment. 
One the other side,  we apply the G4MT to derive the CLT for the  spiked eigenvalues of generalized Fisher matrices under our relaxed assumptions. So that we can get rid of the diagonal assumption or the diagonal  block independent  assumption for the  matrix $\Sigma_1\Sigma_2^{-1}$ and replace them with more general ones. 
It makes sense in practical terms, because  it permits that  the spiked eigenvalues may be generated from the variables partially  dependent on the ones corresponding to the non-spiked eigenvalues. 
With  the general form of the Fisher matrix,  we also can provide
   a few pairs of thresholds for bulks of spiked eigenvalues. 
   Furthermore, the population spiked eigenvalues of  
   the Fisher matrices  in our work are allowed to be infinity. These relaxed conditions  make the results
   more applicable to a wider usage and closer to the actual situation.
   

The rest of the paper is organized as follows. The focused problem  is described   and some preliminaries are prepared in Section~\ref{Pre}; Then, it is  stated formally that our main results in Section~\ref{New}, including the G4MT for 
bulks of spiked eigenvalues of high-dimensional generalized Fisher matrices,  and its application to the CLT for spiked eigenvalues of a generalized spiked Fisher matrix.
 Simulation study  are provided in Section~\ref{Sim}. 
 Finally,  we 
 sketch the main ideas of the proofs in the Supplement.


\section{Problem Description and  Preliminaries}  \label{Pre}

Consider $\Sigma_1^{1 \over 2}\bX$ and  $\Sigma_2^{1 \over 2}\bY$ as two random samples from two independent $p-$ dimensional population, where
\[\bX=(\bx_1,\cdots,\bx_{n_1})=\left(x_{ij}\right), 1\leq i \leq p, ~ 1\leq j \leq n_1\]
and 
\[\bY=(\by_1,\cdots,\by_{n_2})=\left(x_{il}\right), 1\leq i \leq p, ~ 1\leq l \leq n_2.\]
are two independent  $p$-dimensional arrays with  components having zero mean  and  identity variance.
Then $\Sigma_1$ and $\Sigma_2$ are  the relevant population  covariance matrices.

Define
\[T_p=\Sigma_1^{1 \over 2}\Sigma_2^{-{1 \over 2}}\]
and assume that the spectrum of  $T_p^*T_p$ is listed in descending order as  in  
(\ref{array}), some of which
are the population spiked  eigenvalues  lined arbitrarily in groups among all the eigenvalues. 
Denote these spiked  eigenvalues as 
$\alpha_1, \cdots, \alpha_K$ with multiplicity $m_k, k=1,\cdots,K$, respectively, satisfying $m_1+\cdots+m_K=M$, a fixed integer.

 Define
the corresponding  sample covariance matrices of the two observations, {\rm i.e.}
\begin{equation}S_1
=\Sigma_1^{1 \over 2}
\left(\frac{1}{n_1}\bX\bX^*\right)\Sigma_1^{1 \over 2}\label{S1}
\end{equation}
and
\begin{equation}S_2
=\Sigma_2^{1 \over 2}
\left(\frac{1}{n_2}\bY\bY^*\right)\Sigma_2^{1 \over 2},\label{S2}
\end{equation}
respectively.
In this paper, we will investigate  the eigenvalues of the generalized Fisher matrix 
\[\bF=S_1S_2^{-1}=\Sigma_1^{1 \over 2}\tilde {S_1}\Sigma_1^{1 \over 2}\Sigma_2^{-{1 \over 2}}\tilde {S_2}^{-1}\Sigma_2^{-{1 \over 2}}\]     
where 
$\tilde {S_1}=1/n_1\bX\bX^*$ and $\tilde {S_2}={1}/{n_2}\bY\bY^*$  are the standardized sample covariance matrices, respectively. It is well known that the eigenvalues of $\bF$ are the same of the matrix with the form (Still use $\bF$ for brevity, if no confusion):
\begin{equation}
\bF=T_p^{*}\tilde {S_1}T_p\tilde {S_2}^{-1}.\label{F}
\end{equation}
Define the singular value decomposition of $T_p$ as
 \be T_p= V\left(
\begin{array}{cc}
 D_1^{1 \over 2} & \bo    \\
 \bo & D_2^{1 \over 2}     
\end{array}
\right)U^{*}
\label{UDU}
\ee
where $U, V$ are  unitary (orthogonal for complex case) matrices,   $D_1$ is a diagonal matrix of the $M$ spiked eigenvalues of  the generalized spiked Fisher matrix $\bF$
 and $D_2$ is the diagonal matrix of the non-spiked ones  with bounded components.

 Let $J_k$ be the set of  ranks  of  $\alpha_k$ with multiplicity $m_k$ among all the eigenvalues of $T_p^*T_p$, {\rm i.e.}
\[J_k=\{ j_k+1,\cdots, j_k+m_k\}.\]
Set the sample eigenvalues in the descending order  as $\{l_{p,j}(\bA)\}$ 
 for a $p \times p$ matrix $\bA$.
Then,  the  sample eigenvalues of  the generalized spiked  Fisher matrix $\bF$ are also arranged in the descending order as 
\[l_{p,1}(\bF), \cdots,  l_{p,j}(\bF), \cdots,  l_{p,p}(\bF).\]
Therefore, some  assumptions similar to the ones in Jiang and Bai (2018) are presented  in order to figure out the limiting distribution of the spiked eigenvalues of a generalized spiked Fisher matrix $\bF$.

\begin{description}
\item[ Assumption [A\!\!]] The two double arrays $\{x_{ij}, i, j = 1,2,...\} $and $\{y_{ij}, i, j = 1,2,...\}$ consist of independent and identically distributed ({\rm i.i.d.}) random variables with mean 0 and variance 1.
 Furthermore, $E x_{ij}^2=0$ and $E y_{ij}^2=0$  hold for the complex case if the variables and $T_p$ are complex. 
\item[ Assumption~[B\!\!]]  Assume $\lim\limits_{\tau \rightarrow \infty}\tau^4 {\rm P}\left(|x_{11}| >\tau \right)=0$  and
$\lim\limits_{\tau \rightarrow \infty}\tau^4 {\rm P}\left(|y_{11}| >\tau \right)$ $= 0$
for the {\rm i.i.d.} samples, where both of the 4th-moments are not necessarily required to exist.
\item[ Assumption~[C\!\!]] The matrix $T_p=\Sigma_1^{1 \over 2}\Sigma_2^{-{1 \over 2}}$ is non-random and the singular values of $\{T_p^*T_p\}$ are uniformly bounded with at most a finite number of exceptionals. 
Moreover,  the empirical spectral distribution (ESD) of $\{T_p^*T_p\}$,  
$H_n$, tends to proper probability measure $H$  if  $\min( p,n_1,n_2)\rightarrow \infty$.
\item[ Assumption~[D\!\!]]  
 Suppose that 
 \be
 \max\limits_{t,s} |u_{ts}|^2\beta_xI(|x_{11}|<\eta_{n_1}\sqrt{n_1}) \rightarrow 0 \label{CondU1}
 \ee
  \be
 \max\limits_{t,s} |v_{ts}|^2 \beta_yI(|y_{11}|<\eta_{n_2}\sqrt{n_2}) \rightarrow 0 \label{CondU2}
 \ee
  where 
   $\beta_x=\sum\limits_{t=1}^pu_{ts}^4\rE|x_{11}|^4-3$ and $\beta_y=\sum\limits_{t=1}^pv_{ts}^4\rE|y_{11}|^4-3$
   for the considered  the $s$th sample spiked eigenvalue, 
  $u_{ts}$ and $v_{ts}$ are the entries of the matrices $U_1$ and $V_1$, and 
  $U_1, V_1$ are  the first $M$ columns  of matrix $U$ and $V$ defined in (\ref{UDU}), respectively.
   If the matrix $\Sigma_1\Sigma_2^{-1}$ is a diagonal matrix,  it is obvious  that $\beta_x=\rE|x_{11}|^4-3$ and $\beta_y=\rE|y_{11}|^4-3$.
\end{description}


\begin{description}
\item[ Assumption~[E\!\!]]  
Assuming that
\be
 c_{n_1}=p/n_1 \in (0, \infty)\quad \text{and} \quad  c_{n_2} =p/n_2 \in (0,1),\label{limitS}
 \ee
and  $\min( p,n_1,n_2)\rightarrow \infty$ is considered throughout the paper.
Then the spiked eigenvalues of the matrix  $\bF$, $\alpha_1,\cdots, \alpha_K$,  with multiplicities $m_1,\cdots,$  $m_K$  laying out side the support of $H$, satisfy  $\psi'(\alpha_k)>0, $  for $1\le k \le K$, where 
\be
\psi(\alpha_k)=\frac{\alpha_k\left(1-c_{1}\int \frac{t}{t-\alpha_k}\mbox{d}H(t)\right)}{1+c_{2}\int\frac{\alpha_k}{t-\alpha_k}\mbox{d}H(t)}.\label{psik}
\ee
is the phase transition of spiked eigenvalues of generalized spiked Fisher matrix  provided in \cite{Houetal2019}.
 \end{description}

\subsection{Phase Transition of Generalized Spiked Fisher matrix}  \label{Sec2.1}

In this respect, 
 \cite{WangYao2017} proposed the phase transitions of a simplified Fisher matrix, which assumes that $T_p^*T_p$ 
 has a
diagonal block structure; that is,
 the spiked eigenvalues are generated by  random variables independent on the ones for  non-spiked eigenvalues, 
which  is not common  in practice.
 For a more general case, 
  the phase transitions of spiked eigenvalues of generalized spiked Fisher matrix are provided by \cite{Houetal2019}, 
which extends the result of  \cite{WangYao2017}  to a more general case that the matrix $T_p^*T_p$ is  arbitrary symmetric nonnegative definite and both of the spiked eigenvalues  and the 4th moments   may 
not necessarily required to be bounded, meeting the actual cases better. The details are depicted as follows:
   for each spiked eigenvalue 
$\alpha_k$ with multiplicity $m_k, k=1,\cdots,K$ associated  with sample eigenvalues 
$\{l_{j}(\bF), j \in J_k\}$, we have
\begin{proposition}\label{PT}
Under the Assumption~$\bA \sim \bE$, 
  the generalized spiked Fisher matrix $\bF = S_1S_2^{-1}$ is defined in (\ref{F}) with the sample covariance matrices $S_1$ and $S_2$ given in (\ref{S1})-(\ref{S2}).
For any spiked eigenvalue $\alpha_k, (k=1,\cdots,K)$, let
  \[\rho_k =
\left\{
\begin{array}{cc}
  \psi(\alpha_k),& \mbox{ if } \psi'(\alpha_k)>0;  \\
   \psi(\underline\alpha_k), &   \mbox{ if there exists } \underline\alpha_k \mbox{ such that } \psi'(\underline\alpha_k)=0
   \\ &\mbox{ and }\psi'(t)<0,  \mbox{ for all } \alpha_k\le t<\underline{\alpha}_k;
\\
   \psi(\overline\alpha_k), &   \mbox{ if there exists } \overline\alpha_k \mbox{ such that } \psi'(\underline\alpha_k)=0\\
   & \mbox{ and }
 \psi'(s)<0, \mbox{ for all } \overline{\alpha}_k<s\le\alpha_k,
\end{array}
\right.
\]
where $\psi(\alpha_k)$ is defined in (\ref{psik}).
Then, 
 it holds that for all $ j \in J_k$, $\{l_{p,j}\}$ almost surely that  $\{l_{p,j}/\rho_k-1\}$  converges to 0.
 \end{proposition}
 
  \begin{remark}
 Since the convergence of $c_{n_1} \to c_1$, $c_{n_2} \to c_2$  and $H_n \to H$ may be very slow, the difference $\sqrt{n}(l_{p,j} -\psi_k)$ may not have a limiting distribution. Furthermore, from a view of statistical inference, $H_n$ can be treated as the subject population, and $c_{n_1}, c_{n_2}$ can be viewed as the ratio of dimension to sample sizes for the subject sample. So, we usually use 
  \be
  \psi_n(\alpha_k)=\frac{\alpha_k\left(1-c_{n_1}\int \frac{t}{t-\alpha_k}\mbox{d}H_n(t)\right)}{1+c_{n_2}\int\frac{\alpha_k}{t-\alpha_k}\mbox{d}H_n(t)},\label{phink}
  \ee
   instead of $\psi_k$ in $\rho_k$,
and $n$ denotes $(n_1,n_2)$, especially the case of CLT. Then, 
we only require $c_{n_1}=p/n_1; c_{n_2}=p/n_2$, and both the dimensionality $p$ and the sample sizes $n_1, n_2$ grow to infinity  simultaneously, but not necessarily  in proportion. Moreover, the approximation  that
$\{ {l_{j}}/{\rho_k}-1\}$ almost surely converges to 0 still holds  for all $ j \in J_k$.
 \label{RP1}
 \end{remark}



\section{Main Results}  \label{New}

In this section, two facts are going to be proved. The first is the G4MT for generalized Fisher matrix,  which states that 
the limiting distribution of   spiked  eigenvalues of a generalized spiked Fisher matrix is independent of the actual distributions of two samples provided to satisfy the Assumptions $\bA \sim {\bf E}$. The second is the CLT for spiked eigenvalues of a generalized spiked Fisher matrix $\bF$, which can be equivalently obtained by two independent $p$-dimensional Gaussian samples 
 by the G4MT for generalized Fisher matrix.
Before  we start, some explanations of the truncation procedure are given as following.

 \subsection{Truncation and Centralization}  \label{Sec3.1}

Since  $\tau=\eta\sqrt{n_2} \rightarrow \infty$ for every fixed $\eta>0$, it follows  by the Assumption~$B$ that
 \[\eta^4n_2^2 {\rm P} \left(|y_{11}|>\eta\sqrt{n_2}\right) \rightarrow 0,\]
 Therefore, 
 \[n_2^2 {\rm P} \left(|y_{11}|>\eta\sqrt{n_2}\right) \rightarrow 0,\]
  Hence,  there exist a sequence $\eta_n \rightarrow 0$ such that 
  \be
  n_2^2 {\rm P} \left(|y_{11}|>\eta_n\sqrt{n_2}\right) \rightarrow 0,
  \label{n2p}
  \ee
 by Lemma~15 proved in \cite{LiBaiHu2016}.
  
  Let 
  \[\hat y_{ij}= y_{ij} {\rm I}(y_{ij} < \eta_n \sqrt{n_2}) \quad  \text{and}  \quad   \tilde y_{ij}=(\hat y_{ij} - \rE \hat y_{ij})/ \sigma_{n_2}\]
    with $\sigma_{n_2}^2=$ $\rE \left|\hat y_{ij}- \rE  \hat y_{ij} \right|^2$.
    Similar to the proofs of Section~A.1 in Jiang and Bai (2018), it can be illustrated that   the equivalence of replacement of
   the entries of $y_{ij}$ by  the truncated and centralized variables $\tilde y_{ij}$ under the condition  (\ref{n2p}). In addition, 
the convergence rates of arbitrary moments of  $\ \tilde y_{ij}$ are also the same as the one depicted in Lemma~A.1 in \cite{JiangBai2018}.
We can similarly truncate and normalize the entries of $\bX$ without alerting the limiting properties of eigenvalues of $\bF$. 
Therefore, it is reasonable to consider the generalized Fisher-matrix $\bF=S_1S_2^{-1}$ generated from the entries truncated at $\eta_n \sqrt{n_1}$ for $x_{ij}$ and   $\eta_n \sqrt{n_2}$ for $y_{ij}$, centralized and renormalized. For simplicity, 
we assume that $|x_{ij}| <\eta_n \sqrt{n_1}, |y_{ij}| <\eta_n \sqrt{n_2}$, $\rE x_{ij}=\rE y_{ij}=0, \rE |x_{ij}|^2=\rE |y_{ij}|^2=1$ for the real case and 
  Assumption~$\bB$ is satisfied. 
 But for the complex case, the truncation and renormalization cannot reserve the requirement of
$\rE x_{ij}^2=\rE y_{ij}^2 =0$. However, one may prove that  $\rE x_{ij}^2 =o(n_1^{-1})$ and $\rE y_{ij}^2 =o(n_2^{-1})$.

\subsection{Generalized  Four Moment Theorem for Generalized Fisher Matrix  and Its Applications}  \label{Sec3.2}

To facilitate the reading and understanding, the G4MT is introduced in the process of its application to the CLT for the spiked eigenvalues of a generalized  Fisher matrix.
The proof of G4MT will be postponed to  
the Supplement  for the consistency of reading. 

As mentioned in Proposition~\ref{PT}, a packet of $m_k$ consecutive sample eigenvalues $\{l_{p,j}(\bF), j \in J_k\}$ converge to a limit $\rho_k$ laying outside the supporting of the limiting spectral distribution (LSD), $F^{c_1,c_2}$, of $\bF$.
Recall the CLT for the $m_k$-dimensional vector 
\[\Big(\sqrt{p} \big(l_{p,j}(\bF)- \phi(\alpha_k) \big), j \in J_k\Big)\]
given in  \cite{WangYao2017},  
where $ \phi(\alpha_k)=\displaystyle\frac{\alpha_k(\alpha_k+c_1-1)}{\alpha_k-c_2\alpha_k-1}$,  being  a special case of $\psi(\alpha_k)$ in (\ref{psik})  under their 
 assumption of diagonal block independence.
 In the present work, we consider a more general case that  
 the  matrix $T_p^*T_p$ has an arbitrary form as a symmetric nonnegative definite matrix without diagonal block independence, and both of the spiked eigenvalues and the population 4th moments may be allowed to tend to infinity.  Then,  the renormalized random vector   
 \be
(\gamma_{kj}, {j \in J_k}):=~\left(\sqrt{p} \Big(\frac{l_{p,j}(\bF)}{\psi_n(\alpha_k)}-1\Big) , j \in J_k \right),\label{mrv}
\ee
is considered,
where $\psi_n(\alpha_k)$  is used instead of $\psi(\alpha_k)$ because 
 the difference between $l_{p,j}(\bF)$ and $\psi(\alpha_k)$ may  converge very slowly as mentioned in Remark~\ref{RP1}. 

Furthermore,  the CLT for the  renormalized random vector   $(\gamma_{kj}, {j \in J_k})$ is going to be introduced first
 in the following Theorem~\ref{CLT},  which can be seen as an application of G4MT
 for generalized spiked Fisher matrix. The G4MT for generalized spiked Fisher matrix is presented in the process of the proof of Theorem~\ref{CLT}.
 Since the G4MT for generalized spiked Fisher matrix shows  the universality  for the bulks  of spiked eigenvalues of the generalized Fisher matrices, 
the CLT is suitable for a wider usage, including the release of  the 4th-moment constrain and diagonal blocks  assumption of  $T_p^*T_p$. It makes sense in practice that  the spiked eigenvalues are not necessarily required to be independent of the non-spiked ones by eliminating  diagonal block assumption. 

\begin{theorem}
 Suppose that  Assumptions $\bA \sim {\bf E}$ hold.
 For each distant generalized spiked eigenvalue $\alpha_k$\footnote{distant spiked eigenvalue is defined by $\psi'(\alpha)>0$, see Bai and Yao (2012).} with multiplicity $m_k$, the $m_k$- dimensional real vector 
\[\gamma_k=(\gamma_{kj}):=\left(\sqrt{p}\Big(\frac{l_{p,j}(\bF)}{\psi_{n,k}}-1\Big) , j \in J_k \right)\]
converges  weakly to the joint  distribution  of the $m_k$ eigenvalues of Gaussian random matrix 
\[-\frac1{\kappa_s}\left[\Omega_{\psi_{k}}\right]_{kk}\]
where $\psi_{k}:= \psi(\alpha_k)$,  $\psi_{n,k}:= \psi_n(\alpha_k)$ in (\ref{phink}), and
\be
\kappa_s =1+c_{2}\psi^2_{k}m_2(\psi_{k}) +2 c_2\psi_{k}m(\psi_{k})+\alpha_k\psi_{k}\um_2(\psi_{k})+\alpha_k \um(\psi_{k}),\label{ks}
\ee
 $m,\um, m_2, \um_2$ are defined in (\ref{um2}). 
Here  $\psi_k^2m_2(\psi_k)$ is the limit of $\psi_{n,k}^2m_2(\psi_{n,k})$ even if $\alpha_k\to\infty$.
Furthermore, $\Omega_{\psi_{k}}$ is defined in Corollary~{\ref{coro1}}  and 
$\left[\Omega_{\psi_{k}}\right]_{kk}$ is  the $k$th diagonal block of  $\Omega_{\psi_{k}}$   corresponding to the indices $\{i,j \in J_k\}$. 
\label{CLT}
\end{theorem}

 \begin{proof}
First, for the generalized spiked Fisher matrix
$\bF=T_p^{*}\tilde {S_1}T_p\tilde {S_2}^{-1}$,
where  $T_p=\Sigma_1^{1 \over 2}\Sigma_2^{-{1 \over 2}}$, and 
$\tilde {S_1}=1/n_1\bX\bX^*, \tilde {S_2}={1}/{n_2}\bY\bY^*$  are the standardized sample covariance matrices, respectively.
 By singular value decomposition, we have 
 \be
 T_p= U
\left(
\begin{array}{cc}
 D_1^{1 \over 2} & \bo    \\
 \bo & D_2^{1 \over 2}     
\end{array}
\right)V^*
\label{SVD}
\ee
where $V, U$ are orthogonal matrices,   $D_1$ is a diagonal matrix of the $M$ spiked eigenvalues
 and $D_2$ is the diagonal matrix of the non-spiked eigenvalues.
Consider the arbitrary sample spiked eigenvalue of $\bF$, $l_{p,j}, j \in J_k$, 
by the eigenequation with $\bF=T_p^{*}\tilde {S_1}T_p\tilde {S_2}^{-1}$, we have
\[0=|l_{p,j}\bI- \bF|=\left|l_{p,j}\bI-V\left(
\begin{array}{cc}
 D^{1 \over 2}_1 & \bo    \\
 \bo & D^{1 \over 2}_2     
\end{array}
\right)
U^* \tilde {S_1}U
\left(
\begin{array}{cc}
 D^{1 \over 2}_1 & \bo    \\
 \bo & D^{1 \over 2}_2     
\end{array}
\right)V^*\tilde {S_2}^{-1}\right|,\]
which is equivalent to 
\begin{eqnarray*}
0&=&|l_{p,j} V^*\tilde {S_2}V- {\rm diag}(D_1^{1 \over 2},D_2^{1 \over 2})U^* \tilde {S_1} U{\rm diag}(D_1^{1 \over 2},D_2^{1 \over 2})|\\
&=&\Bigg|\begin{pmatrix} l_{p,j} V_1^*\tilde {S_2} V_1&l_{p,j} V_1^* \tilde {S_2} V_2\\  l_{p,j} V_2^*\tilde {S_2} V_1&l_{p,j} V_2^* \tilde {S_2} V_2\end{pmatrix}
-\begin{pmatrix}D_1^{1 \over 2}U_1^* \tilde {S_1}U_1D_1^{1 \over 2}&D_1^{1 \over 2}U_1^* \tilde {S_1}U_2D_2^{1 \over 2}\\ D_2^{1 \over 2}U_2^* \tilde {S_1} U_1 D_1^{1 \over 2}&D_2^{1 \over 2}U_2^* \tilde {S_1}U_2D_2^{1 \over 2}\end{pmatrix}\Bigg|,
\end{eqnarray*}

Since $l_{p,j}$ is  an sample eigenvalue of $\bF$ but not the one of $D_2^{1 \over 2}U_2^* \tilde {S_1}U_2D_2^{1 \over 2}\!(V_2^* \tilde {S_2} V_2\!)\!^{-1}$, then the following equation holds
\begin{align*}
0&=\Big|l_{p,j} V_1^*\tilde {S_2} V_1-D_1^{1 \over 2}U_1^* \tilde {S_1}U_1D_1^{1 \over 2}-(l_{p,j} V_1^* \tilde {S_2} V_2-D_1^{1 \over 2}U_1^* \tilde {S_1}U_2D_2^{1 \over 2})\times \\
&\! \times \! (l_{p,j} V_2^* \tilde {S_2} V_2-D_2^{1 \over 2}U_2^* \tilde {S_1}U_2D_2^{1 \over 2})^{-1}(l_{p,j} V_2^*\tilde {S_2} V_1-D_2^{1 \over 2}U_2^* \tilde {S_1} U_1 D_1^{1 \over 2})\Big|\\
&=\Big|l_{p,j} V_1^*\tilde {S_2} V_1-D_1^{1 \over 2}U_1^* \tilde {S_1}U_1D_1^{1 \over 2}-(l_{p,j} V_1^* \tilde {S_2} V_2-D_1^{1 \over 2}U_1^* \tilde {S_1}U_2D_2^{1 \over 2})Q^{-{1 \over 2}}\times \\
&\! \times \!(l_{p,j} \bI_{p-M}\!-\!Q^{-{1 \over 2}}D_2^{1 \over 2}\!U_2^* \tilde {S_1}U_2D_2^{1 \over 2}Q^{-{1 \over 2}})^{-1}Q^{-{1 \over 2}}(l_{p,j} V_2^*\tilde {S_2} V_1\!-\!D_2^{1 \over 2}U_2^* \tilde {S_1} U_1 D_1^{1 \over 2})\Big|
\end{align*}
where $Q=V_2^* \tilde {S_2} V_2$. By in-out-exchanging formula, 
\[Z(Z'Z-\lambda \bI)^{-1}Z'=\bI+\lambda (ZZ'-\lambda \bI)^{-1},\]
we have 
\begin{align*}
&D_1^{1 \over 2}U_1^* \tilde {S_1}U_1D_1^{1 \over 2}\!+\!D_1^{1 \over 2}U_1^* \tilde {S_1}U_2D_2^{1 \over 2}Q^{-{1 \over 2}}
(l_{p,j} \bI_{p-M}\!-\!Q^{-{1 \over 2}}\!D_2^{1 \over 2}U_2^* \tilde {S_1}U_2D_2^{1 \over 2}Q^{-{1 \over 2}})^{-1}\\
&\!\!\times\!\! Q\!^{-\!{1 \over 2}}\!D_2^{1 \over 2}U_2^* \tilde {S_1} U_1 D_1^{1 \over 2}
\!=\!\frac{l_{p,j}}{n_1} D_1^{1 \over 2}U_1^*\bX(l_{p,j}\bI_{n_1}\!-\!\frac1{n_1}\bX^*U_2 D_2^{1 \over 2}Q^{-1}\!D_2^{1 \over 2} U_2^* \bX)\!^{-1}\!\bX\!^*U_1D_1\!^{1 \over 2}
\end{align*}
Then,  define
$$
 \tilde \bF=\displaystyle\frac 1 {n_1}Q^{-{1 \over 2}}D_2^{1 \over 2}U_2^* \bX\bX^*U_2D_2^{1 \over 2}Q^{-{1 \over 2}}, \quad \underline{\tilde \bF}=\displaystyle\frac 1 {n_1}\bX^*U_2D_2^{1 \over 2}  Q^{-1}D_2^{1 \over 2} U_2^* \bX,\label{FFdef}
 $$
the equation above is equivalent to
 \begin{align}
0=\bigg|&
\frac{l_{p,j}}{n_2}V_1^*\bY  \bY^*V_1\non
&-\frac{l_{p,j}^2}{n_2^2}V_1^*\bY\bY^*V_2Q^{-{1 \over 2}} (l_{p,j} \bI_{p-M}- \tilde \bF\big)^{-1} Q^{-{1 \over 2}}V_2^*\bY \bY^*V_1\non
&-\frac{l_{p,j}}{n_1} D_1^{1 \over 2}U_1^*\bX\big(l_{p,j}\bI_{n_1}-\underline{\tilde \bF}\big)^{-1}\bX^*U_1D_1^{1 \over 2}\non
&+\frac{l_{p,j}}{n_2}V_1^*\bY \bY^*V_2 Q^{-{1 \over 2}}\big(l_{p,j} \bI_{p-M}- \tilde \bF\big)^{-1} Q^{-{1 \over 2}}
\frac{1}{n_1}D_2^{1 \over 2}U_2^*\bX\bX^* U_1 D_1^{1 \over 2}\non
&+\frac{l_{p,j}}{n_1}D_1^{1 \over 2}U_1^* \bX\bX^*U_2D_2^{1 \over 2}Q^{-{1 \over 2}} (l_{p,j} \bI_{p-M}- \tilde \bF)^{-1}Q^{-{1 \over 2}}\frac{1}{n_2} V_2^*\bY\bY^* V_1\bigg|
\label{eieqn}
\end{align}

 Furthermore, define
 and  
 $m (\lambda), \um(\lambda), m_2(\lambda), \um_2(\lambda), m_3(\lambda)$  as below:
   \bqa
  &m(\lambda)=\displaystyle\int \frac{1}{x-\lambda} \md \tilde {F}(x),  \quad &  \underline{m}(\lambda)=\displaystyle\int \frac{1}{x-\lambda} \md  \underline{\tilde F}(x) ;\non
 & m_2(\lambda)=\displaystyle\int \frac{1}{(\lambda-x)^2} \md \tilde {F}(x),   \quad & 
    \underline{m}_2(\lambda)=\displaystyle\int \frac{1}{(\lambda-x)^2} \md \underline{\tilde F}(x) ;\non
 &m_3(\lambda)=\displaystyle\int \frac{x}{(\lambda-x)^2} \md \tilde {F}(x),  \quad & \label{um2}
  \eqa
where $\tilde F(x)$   and $\underline{ \tilde F}(x)$  are the LSDs of  the matrices $\tilde \bF$ and $\underline{ \tilde \bF}$, 
respectively.
Since $(U_1^*X, V_1^*Y)$ is independent of $(U_2^*X, V_2^*Y)$, and the covariance matrix between $U_1^*X$ and  $V_1^*Y$ is a zero matrix $\bo_{M\times M}$. Then 
 \begin{align}
&\frac{l_{p,j}}{n_2}V_1^*\bY \bY^*V_2 Q^{-{1 \over 2}}\big(l_{p,j} \bI_{p-M}- \tilde \bF\big)^{-1} Q^{-{1 \over 2}}
\frac{1}{n_1}D_2^{1 \over 2}U_2^*\bX\bX^* U_1 D_1^{1 \over 2} \rightarrow \bo_{M\times M}\label{eq01}\\
&\frac{l_{p,j}}{n_1}D_1^{1 \over 2}U_1^* \bX\bX^*U_2D_2^{1 \over 2}Q^{-{1 \over 2}} (l_{p,j} \bI_{p-M}- \tilde \bF)^{-1}Q^{-{1 \over 2}}\frac{1}{n_2} V_2^*\bY\bY^* V_1
\rightarrow \bo_{M\times M}\label{eq02}
\end{align}
According to   Lemma~2.7 in \cite{BaiSilverstein1998}, we have  
 \begin{align}
&\frac{l_{p,j}}{n_2}V_1^*\bY  \bY^*V_1 \rightarrow 
\psi_k \bI_M\\
-&\frac{l_{p,j}^2}{n_2^2}V_1^*\bY\bY^*V_2Q^{-{1 \over 2}} (l_{p,j} \bI_{p-M}- \tilde \bF\big)^{-1} Q^{-{1 \over 2}}V_2^*\bY \bY^*V_1 \rightarrow c_2\psi_k^2m(\psi_k)\bI_M\\
-&\frac{l_{p,j}}{n_1} D_1^{1 \over 2}U_1^*\bX\big(l_{p,j}\bI_{n_1}-\underline{\tilde \bF}\big)^{-1}\bX^*U_1D_1^{1 \over 2} \rightarrow \psi_k\um(\psi_k) D_1
\label{eq05}
\end{align}
Therefore, combine the equations (\ref{eieqn}), (\ref{eq01}) - (\ref{eq05}), we  obtain that $\psi_k$ satisfies the following equation
\begin{equation}
\psi_{k}+c_{2}\psi^2_{k}m(\psi_{k})+\psi_k\um(\psi_{k}) \alpha_k=0.\label{0eqa}
\end{equation}

Define 
\begin{align}
 \Omega_{M}(\lambda,\bX,\bY)&= \Omega_{M,1}(\lambda,\bX,\bY)+
\Omega_{M,2}(\lambda,\bX,\bY)+\Omega_{M,3}(\lambda,\bX,\bY)\non
 &\quad+\Omega_{M,4}(\lambda,\bX,\bY)+\Omega_{M,5}(\lambda,\bX,\bY),\label{OmegaM}
\end{align}
where
\begin{align}
 \Omega_{M,1}(\lambda,\bX,\bY)&=\sqrt{p} V_1^*( {\tilde S}_{2}-\bI_p)V_1\label{OmegaM}\\
  \Omega_{M,2}(\lambda,\bX,\bY)&=\frac{\sqrt{p}\lambda}{n_2}{\rm tr}\! (\lambda \bI_{p-M}-\tilde \bF)^{-1} \bI_M\non
 &-\frac{\sqrt{p}\lambda}{n_2^{2}}
 V_1^*\bY\bY^*V_2Q^{-\frac1{2}} (\lambda \bI_{p-M}-\tilde \bF)^{-1} Q^{-\frac1{2}}V_2^*\bY\bY^*V_1\non
\Omega_{M,3}(\lambda,\bX,\bY)&=\frac{\sqrt{p}}{n_1}{\rm tr}\! (\lambda \bI_{n_1}\!\!-\! \underline{\tilde \bF})^{-1}  D_1 \!-\!\frac{\sqrt{p}}{n_1}D_1^{1 \over 2}U_1^*\bX\big(\lambda\bI_{n_1}\!\!-\! \underline{\tilde \bF}\big)^{-1}\!\bX^*U_1D_1^{1 \over 2}\non
\Omega_{M,4}(\lambda,\bX,\bY)&= \frac{\sqrt{p}}{n_1n_2}V_1^*\bY\bY^*V_2 Q^{-{1 \over 2}}\big(\lambda \bI_{p-M}\!-\!\tilde \bF\big)^{-1} Q^{-{1 \over 2}}
D_2^{1 \over 2}U_2^*\bX\bX^* U_1 D_1^{1 \over 2}\non
 \Omega_{M,5}(\lambda,\bX,\bY)&= \frac{\sqrt{p}}{n_1n_2}D_1^{1 \over 2}U_1^* \bX\bX^*U_2D_2^{1 \over 2} Q^{-{1 \over 2}}\big(\lambda \bI_{p-M}-\tilde \bF\big)^{-1} Q^{-{1 \over 2}}
V_2^*\bY\bY^* V_1.\nonumber
\end{align}
For every sample spiked eigenvalue, $l_{p,j}, j \in J_i, i=1,\cdots, K$ and non-zero population spiked eigenvalues, it follows from equation (\ref{eieqn}) that
\begin{align}
0 &=\bigg|l_{p,j}\bI_M\!\!-\!\!\frac{l^2_{p,j}}{n_2} {\rm tr}(l_{p,j} \bI_{p\!-\!M}\!-\! \tilde \bF)^{-1}\! \bI_{M}
\!\!-\!\!\frac{l_{p,j}}{n_1}{\rm tr}(l_{p,j} \bI_{n_1}\!\!-\! \underline{\tilde \bF})^{-1} \! D_1
\!\!+\!\!\frac{l_{p,j}}{\sqrt{p}}\Omega_{M}(l_{p,j},\!\bX,\bY)\bigg| \non
&=\bigg|(\psi_{n,k}+c_{2}\psi^2_{n,k}m(\psi_{n,k}))\bI_M+\psi_{n,k}\um(\psi_{n,k})D_1+\frac 1{\sqrt{p}}\gamma_{kj}\psi_{n,k}\bI_M\non
&\quad+B_1(l_{p,j})+B_2(l_{p,j})+\frac{\psi_{n,k}}{\sqrt{p}}\Omega_{M}(\psi_{n,k},\bX,\bY)+o(\frac 1{\sqrt{p}})\bigg| 
\label{IBJ}
\end{align}
where  the involved $B_i(l_{p,j}), i=1,2$ are specified as following, and
 $\psi_{n,k}$  is used instead of $\psi_k$  to  avoid  the slowly convergence 
as mentioned  in Remark~\ref{RP1}. In details, 
\begin{align}
&\quad B_{1}(l_{p,j})\non
&= \frac{\psi^2_{n,k}}{n_2^{2}}
 V_1^*\bY\bY^*V_2Q^{-\frac1{2}} (\psi_{n,k} \bI_{p-M}-\tilde \bF)^{-1} Q^{-\frac1{2}}V_2^*\bY\bY^*V_1\non
&\quad-\frac{l^2_{p,j}}{n_2^{2}}
 V_1^*\bY\bY^*V_2Q^{-\frac1{2}} (l_{p,j} \bI_{p-M}-\tilde \bF)^{-1} Q^{-\frac1{2}}V_2^*\bY\bY^*V_1\non
 &= \frac{\psi^2_{n,k}}{n_2} V_1^*\bY\!\bigg(\!\frac{1}{n_2}\!\bY^*V_2Q^{-\frac1{2}}\! \Big(\!(\psi_{n,k} \bI_{p-M}\!-\!\tilde \bF)^{-1} 
\! \!-(l_{p,j} \bI_{p-M}\!-\!\tilde \bF)^{-1}\!\Big)\!Q^{-\frac1{2}}\!V_2^*\bY\!\bigg)\!\bY^*V_1\non
& \quad-\frac{l^2_{p,j}-\psi^2_{n,k}}{n_2}  V_1^*\bY\Big(\frac{1}{n_2}\bY^*V_2Q^{-\frac1{2}} (l_{p,j} \bI_{p-M}-\tilde \bF)^{-1} Q^{-\frac1{2}}V_2^*\bY\Big)\bY^*V_1\non
&=\!\frac 1{\sqrt{p}}\!\gamma_{kj}\psi^3_{n,k}\frac1{n_2}  
  \!\tr\left(\!\Big(\psi_{n,k}\bI_{p-M}\!-\! \tilde \bF \Big)^{-1}\!\!\Big((\psi_{n,k}+\frac 1{\sqrt{p}}\gamma_{kj}\psi_{n,k}) \bI_{p-M}\!\!-\!\tilde \bF \Big)^{-1}
\!\right)\!\bI_{M} \non 
&\quad-\!\frac 1{\sqrt{p}}\!\gamma_{kj}2\psi^2_{n,k}\frac1{n_2}  
  \!\tr\Big(l_{p,j}\bI_{p-M}\!-\! \tilde \bF \Big)^{-1}\bI_{M}+o(\frac 1{\sqrt{n_2}}) \non
&=\!\frac 1{\sqrt{p}}\!\gamma_{kj}\Big( c_{2}\psi^3_{n,k}m_2(\psi_{n,k}) +2 c_2\psi^2_{n,k}m(\psi_{n,k})\Big)\bI_{M}+
o(\frac 1{\sqrt{p}}) 
  \label{Bcha1}
\end{align}
and 
\begin{align}
&\quad B_{2}(l_{p,j})\non
&=\frac{\psi_{n,k}}{n_1}D_1^{1 \over 2}U_1^*\bX\big(\psi_{n,k}\bI_{n_1}\!-\! \underline{\tilde \bF}\big)\!^{-1}\!\bX^*U_1D_1^{1 \over 2}\!-\!\frac{l_{p,j}}{n_1}D_1^{1 \over 2}U_1^*\bX\big(l_{p,j}\bI_{n_1}\!-\! \underline{\tilde \bF}\big)^{-1}\bX^*U_1D_1^{1 \over 2}\non
&=\frac{\psi_{n,k}}{n_1}D_1^{1 \over 2}U_1^*\bX\Big((\psi_{n,k} \bI_{n_1}-\underline{\tilde \bF})^{-1} 
 -(l_{p,j} \bI_{n_1}-\underline{\tilde \bF})^{-1}\Big)
\bX^*U_1D_1^{1 \over 2}\non
&\quad-\frac{l_{p,j} -\psi_{n,k}}{n_1}D_1^{1 \over 2}U_1^*\bX(l_{p,j} \bI_{n_1}-\underline{\tilde \bF})^{-1}
\bX^*U_1D_1^{1 \over 2}\non
&=\!\frac 1{\sqrt{p}}\!\gamma_{kj}\psi^2_{n,k}\frac1{n_1}   \tr\big((\psi_{n,k} \bI_{n_1}-\underline{\tilde \bF})^{-2} \big)D_1\non
&\quad
-\frac 1{\sqrt{p}}\!\gamma_{kj}\psi_{n,k}\frac1{n_1}   \tr\big((\psi_{n,k} \bI_{n_1}-\underline{\tilde \bF})^{-1} \big)D_1
+o(\frac 1{\sqrt{n_1}})\non
&=\!\frac 1{\sqrt{p}}\!\gamma_{kj}\Big(\psi^2_{n,k}\um_2(\psi_{n,k})+\psi_{n,k} \um(\psi_{n,k})
\Big)D_1+
o(\frac 1{\sqrt{p}}) 
  \label{Bcha2}
\end{align}

Then, combine all  of the equations (\ref{IBJ}), (\ref{Bcha1}) and (\ref{Bcha2}), for non-zero spiked eigenvalues and $\psi_{n,k}$, it is obvious that 
\begin{align}
0&=\bigg|(\psi_{n,k}+c_{2}\psi^2_{n,k}m(\psi_{n,k}))\bI_M+\psi_{n,k}\um(\psi_{n,k})D_1\non
& +\frac 1{\sqrt{p}}\gamma_{kj}\bigg(
\big( \psi_{n,k}+c_{2}\psi^3_{n,k}m_2(\psi_{n,k}) +2 c_2\psi^2_{n,k}m(\psi_{n,k})\big)\bI_{M}\non
&+\big(\psi^2_{n,k}\um_2(\psi_{n,k})+\psi_{n,k} \um(\psi_{n,k})\big)D_1\bigg)+\frac{\psi_{n,k}}{\sqrt{p}}\Omega_M(\psi_{n,k},\bX,\bY)+o(\frac1{\sqrt{p}}) \bigg|.\nonumber
\label{eigeneq8} 
\end{align}

Moreover, $\psi_{n,k}$ satisfies the equation (\ref{0eqa}), it means that the population spiked eigenvalues $\alpha_u$ in the $u$-th diagonal block of $D_1$ makes
$\psi_{n,k}+c_{2} \psi^2_{n,k}m(\psi_{n,k}) + \psi_{n,k}\um(\psi_{n,k}) \alpha_u$
keep away from 0,if $u\neq k$;
and satisfies 
$\psi_{n,k}+c_{2}\psi^2_{n,k}m(\psi_{n,k})+\psi_{n,k}\um(\psi_{n,k}) \alpha_k=0$.  For non-zero limit of spiked eigenvalue, $\psi_{n,k}$, each $k$-th diagonal block of the above equation is  multiplied $p^{\frac{1}{4}}$  by rows and columns, respectively. By Lemma~4.1 in \cite{BaiMiaoRao1991},
we obtain that 
\begin{align}
&\bigg|\gamma_{kj}
\psi_{n,k}\big( 1+c_{2}\psi^2_{n,k}m_2(\psi_{n,k}) +2 c_2\psi_{n,k}m(\psi_{n,k})+\alpha_k\psi_{n,k}\um_2(\psi_{n,k})+\alpha_k \um(\psi_{n,k})\big)\bI_{m_k}\non
&+\psi_{n,k}\big[\Omega_M(\psi_{n,k},\bX,\bY)\big]_{kk}+o(1) \bigg|=0,
\end{align}
where $\left[~\cdot ~\right]_{kk}$ is  the $k$-th diagonal block of  a matrix  corresponding to the indices $\{i,j \in J_k\}$. 
According to the  Skorokhod strong representation in  \cite{Skorokhod1956, HuBai2014}, 
it follows  that the  convergence 
of $\Omega_M(\psi_{n,k},\bX,\bY)$  and (\ref{eigeneq8}) can be achieved simultaneously in probability 1 by choosing an appropriate probability space.


Let
\be
\kappa_s =1+c_{2}\psi^2_{n,k}m_2(\psi_{n,k}) +2 c_2\psi_{n,k}m(\psi_{n,k})+\alpha_k\psi_{n,k}\um_2(\psi_{n,k})+\alpha_k \um(\psi_{n,k}).\label{kappas}
\ee
Thus, it is obvious that, $\gamma_{kj}$ asymptotically satisfies the following equation   
\begin{equation}
\Big|\gamma_{kj} \cdot \kappa_s  \bI_{m_k}
+\big[\Omega_{\psi_{k}}\big]_{kk}
 \Big|  =0,\label{fineq0}
 \end{equation}
 where $\Omega_{\psi_{k}}$ is an $M\times M$ Hermitian matrix, being  the limiting distribution of 
 $\Omega_M(\psi_{n,k},\bX,\bY)$.
 

Therefore, 
our
remaining major work is to derive the  limiting distribution of   $\Omega_M(\psi_{n,k},\bX,\bY)$.
To this end,
the G4MT for generalized Fisher-matrix is proposed 
in the following theorem, which shows that the limiting distribution of  the spiked  eigenvalues of a generalized spiked Fisher matrix is independent of the actual distributions of the samples provided to satisfy the Assumptions $\bA \sim {\bf E}$. For the consistency of reading,  its proof is  postponed to Supplement. 

\begin{theorem}[{\bf Generalized 
Four Moment Theorem}]\label{thm2}
Assuming that $(\bX, \bY)$ and $(\bW,\bZ)$ are two pairs of double arrays, each of which  satisfies Assumptions $\bA \sim {\bf E}$,  then
 $\Omega_{M}(\lambda,\bX,\bY)$ and $\Omega_{M}(\lambda,\bW,\bZ)$ have the same limiting distribution, provided one of them has. 
\end{theorem}
According to  Theorem \ref{thm2}, we may assume that $\bX$ and $\bY$ are consist of entries with {\rm i.i.d.} standard random variables in deriving the limiting distributions of $\Omega_M(\psi_{n,k},\bX,\bY)$.  
 Thus, we have the following Corollary.
 
\begin{corollary}\label{coro1}
Suppose that  both $\bX$ and $\bY$ satisfy the Assumptions $\bA \sim {\bf E}$, and 
let
\begin{align}
&\theta_k=c_2+ c_2^2\psi_{k}^2 m_2(\psi_{k})+2c_2^2\psi_km(\psi_{k})+ c_1\alpha_k^2\um_2(\psi_{k})+ 2c_1 c_2 \alpha_k m_3(\psi_{k}). \label{theta}
\end{align}
Then, it holds that 
 $\Omega_M(\psi_{n,k}, \bX,\bY)$ tends to a limiting distribution of an $M\times M$ Hermitian matrix {$\Omega_{\psi_{k}}$,  where
 $\frac{1}{\sqrt{\theta_k}}\left[\Omega_{\psi_{k}}\right]_{kk}$ is  Gaussian Orthogonal Ensemble (GOE) for the real case, with 
the entries above the diagonal being ${\rm i.i.d.} \mathcal{N}(0,1)$ and the entries on the diagonal being ${\rm i.i.d.} \mathcal{N}(0,2)$. 
For the complex case, the $\frac{1}{\sqrt{\theta_k}}\left[\Omega_{\psi_{k}}\right]_{kk}$ is  GUE, whose  entries  are all ${\rm i.i.d.} \mathcal{N}(0,1)$.}
\end{corollary}
The proof  of Corollary~{\ref{coro1}} is also detailed in Supplement. 

Therefore,  by the equation (\ref{fineq0}), the $m_k$-dimensional real vector 
$\{\gamma_{kj}, j\in J_k\}$ converges  weakly to the distribution  of  the $m_k$ eigenvalues of the Gaussian random matrix 
\[-\frac1{\kappa_s}\left[\Omega_{\psi_{k}}\right]_{kk}\]
for each distant generalized spiked eigenvalue. The   distribution of $\Omega_{\psi_{k}}$ is  detailed in Corollary~\ref{coro1}.
Then, the CLT For each distant  spiked eigenvalue of a
generalized covariance matrix  is obtained.
\end{proof}

\begin{remark}
Let $\tilde c_i= (p-M)/n_i, i=1,2.$
 Since the convergence of $c_{n_1} \to c_1$, $c_{n_2} \to c_2$   may be very slow, and the differences between  $c_{n_i}$ and $\tilde c_{i}$ may also go to 0 slowly. Therefore, in the aspect of  statistical inference, 
 the $m_k$-dimensional real vector 
\[\gamma_k=(\gamma_{kj}):=\left(\sqrt{p-M}\Big(\frac{l_{p,j}(\bF)}{\psi_{n,k}}-1\Big) , j \in J_k \right)\]
is used instead, and all the conclusions of Theorem~\ref{CLT} still holds, but with 
$c_i$ substituted by $\tilde c_i$, i=1,2, except the ones in $\psi_{n,k}$.
 \label{rmk1}
\end{remark}

Actually, our result  cannot cover some exceptional  cases, in which  the Assumption~{\bf D} is not satisfied. For example, it is the case that $T_p^*T_p$ is a diagonal matrix or a diagonal block matrix.
For  such special cases,   we use $X,Y$ to represent the random variables corresponding to the arrays $\{x_{ij}\}$ and $\{y_{ij}\}$, respectively. Then we require that the 4th moments of $X,Y$  and all the spiked eigenvalues are bounded, and obtain the following conclusion, which plays the same role as the result of  \cite{WangYao2017}  involved with  the bounded 4th moments of $X$ and $Y$ under the assumption of diagonal block independence. 

\begin{remark}
Suppose that  both $\bX$ and $\bY$ satisfy the Assumptions $\bA, {\bf B}, {\bf C}$ and  ${\bf E}$, excluding the Assumption ${\bf D}$, but the 4th moments of $X,Y$  and all the spiked eigenvalues are bounded. Then
 all the conclusions of Theorem~\ref{CLT} still holds, but the limiting distribution of  $\Omega_M(\psi_{n,k}, \bX,\bY)$ turns to 
 an $M\times M$ Hermitian matrix $\Omega_{\psi_{k}}=(\omega_{st})$,  which has the independent Gaussian entries of 
 mean 0 and variance 
 \[{\rm Var}(\omega_{st})=
\left\{
\begin{array}{cc}
2\theta_k+\beta_x\nu_1+\beta_y\nu_2 ,  &   s=t \in J_k   \\
\theta_k,  &      s\neq t \in J_k   
\end{array}
\right.
\]
where  $\theta_k$ is defined in (\ref{theta}),  
$\nu_1=c_1\alpha_k^2/\big(\psi_k(1+c_1m(\psi_k))\big)^2$, $\nu_2=c_2\big(1+c_2\psi_k m(\phi_k))\big)^2$
%
 for the real case.
 For the complex case, the $\frac{1}{\sqrt{\theta_k}}\left[\Omega_{\psi_{k}}\right]_{kk}, k=1,\cdots, K$ are still  GUE, whose  entries  are all ${\rm i.i.d.} \mathcal{N}(0,1)$.
 \label{rmk2}
\end{remark}
This remark is used in the simulations of Case I under non-Gaussian assumptions.


\section{ Simulation Study}\label{Sim}

In this section, simulations are provided to  evaluate our main results comparing to the existing work in  \cite{WangYao2017}.
We consider two scenarios:  
\begin{description}
\item[ Case I:] The matrix $T_pT_p^*$ is assumed as a finite-rank perturbation of a identity matrix $\bI_p$, where $\Sigma_2=\bI_p$ and $\Sigma_1$ is an identity matrix with  the spikes $(20,0.2,0.1)$ of the multiplicity $(1,2,1)$ in the descending order and  thus $K=3$ and $M=4$ as proposed in  \cite{WangYao2017}.  
\item[Case II:]  The matrix $T_pT_p^*$ is a general positive definite matrix, but not necessary  with diagonal blocks independence assumption. 
It is designed as below: 
$\Sigma_2=\bI_p$ and $\Sigma_1=U_0 \Lambda U_0^*$, where  $\Lambda$ is a diagonal matrix  made  up of
the spikes   $(20,0.2,0.1)$ with multiplicity $(1,2,1)$ and the other eigenvalues being 1 in the descending order. 
Let $U_0$  be equal to the matrix composed of eigenvectors
 of the following matrix
\begin{align}
\left(
\begin{array}{ccccc}
  1 & \rho & \rho^2 & \cdots & \rho^{p-1} \\
  \rho & 1 & \rho & \cdots & \rho^{p-2} \\
  \ldots & \ldots &  &  &  \\
  \rho^{p-1} & \rho^{p-2} & \cdots & \rho & 1 \\
\end{array}
\right).
\end{align}
where $\rho=0.5$.
\end{description}

 For every scenario, 
we propose  two population assumptions as following:

\begin{description}
\item[ Gaussian Assumption:]  $x_{ij}$ and $y_{ij}$  are both ${\rm i.i.d.}$ sample from  standard Gaussian population; 
\item[ Binomial  Assumption:]  $x_{ij}$ and $y_{ij}$ are  ${\rm i.i.d.}$ samples from the binary variables valued at $\{-1, 1\}$ with equal probability $1/2$, and $\beta_x=\beta_y=1-3=-2$.

%
%
\end{description} 

Then,
we report the empirical  distribution  with 1000 replications   at  the
 values of  $p=200$ and 
 sample sizes $ n_1=1000$ and $n_2=400$. 
 
 \subsection{Case I under Gaussian Assumption}
 As described in {\rm \bf Case~I}, we have the spikes $\alpha_1=20$, $\alpha_2=0.2$ and $\alpha_3=0.1$, First, we assume that  the Gaussian Assumption hold and  let $l_{1,p},\cdots, l_{p,p} $ be the sorted sample eigenvalues of the $F-$matrix defined in (\ref{F}). Then by the Theorem~\ref{CLT}, we obtain the limiting results as below. 
\begin{itemize}
\item First, take the single  population  spikes $\alpha_1=20$ and $\alpha_3=0.1$ into account, and consider the largest sample eigenvalue $l_{1,p}$ , we have :
\[\gamma_1= \sqrt{p-4}\Big(\frac{l_{p,1}(\bF)}{\psi_{n,1}}-1\Big) \rightarrow N(0, \sigma_1^2)\]
where 
\[\psi_{n,1}=42.667;\quad  \sigma_1^2=2.383 .\]

Similarly, for
the least eigenvalues   $l_{p,p}$, we have 
\[\gamma_3= \sqrt{p-4}\Big(\frac{l_{p,p}(\bF)}{\psi_{n,3}}-1\Big) \rightarrow N(0, \sigma_3^2)\]
where 
\[\psi_{n,3}= 0.0737;\quad  \sigma_3^2=1.343.\] 
\item Second, for  the spikes $\alpha_2=0.2$ with multiplicity 2, we consider the sample eigenvalue $l_{1,p-1}$ and
 $l_{1,p-2}$, we obtain that the two-dimensional random vector 
\[\gamma_2=\left(\gamma_{2,1}, \gamma_{2,2}\right)'=\left( \sqrt{p-4}\Big(\frac{l_{p,p-2}(\bF)}{\psi_{n,2}}-1\Big),
\sqrt{p-4}\Big(\frac{l_{p,p-1}(\bF)}{\psi_{n,2}}-1\Big)\right)' \]
converges to the eigenvalues of random matrix $-\frac1{\kappa_s}\left[\Omega_{\psi_{2}}\right]_{22}$,
where  $\psi_{n,2}=0.133$, $\kappa_s=1.441$ for the spike $\alpha_2=0.2$.  Furthermore, the
matrix 
$\left[\Omega_{\psi_{2}}\right]_{22}$
 is a $2\times 2$ symmetric matrix with the independent Gaussian entries, of which the $(s,t)$ element has mean zero and the variance given by 
 \[var(w_{st})=
\left\{
\begin{array}{cc}
 2.326, &~ \text{if}~ s =t  \\
 1.163,  &  ~ \text{if}~ s \neq t  
\end{array}
\right.
\]
 \end{itemize}
The simulated empirical distributions of the spiked eigenvalues from Normal assumption under {\rm \bf Case~I}  are drawn in Figure~1 in contrast to their corresponding limiting distributions.

 \begin{figure}[htbp]
\begin{center}
\includegraphics[width = .37\textwidth]{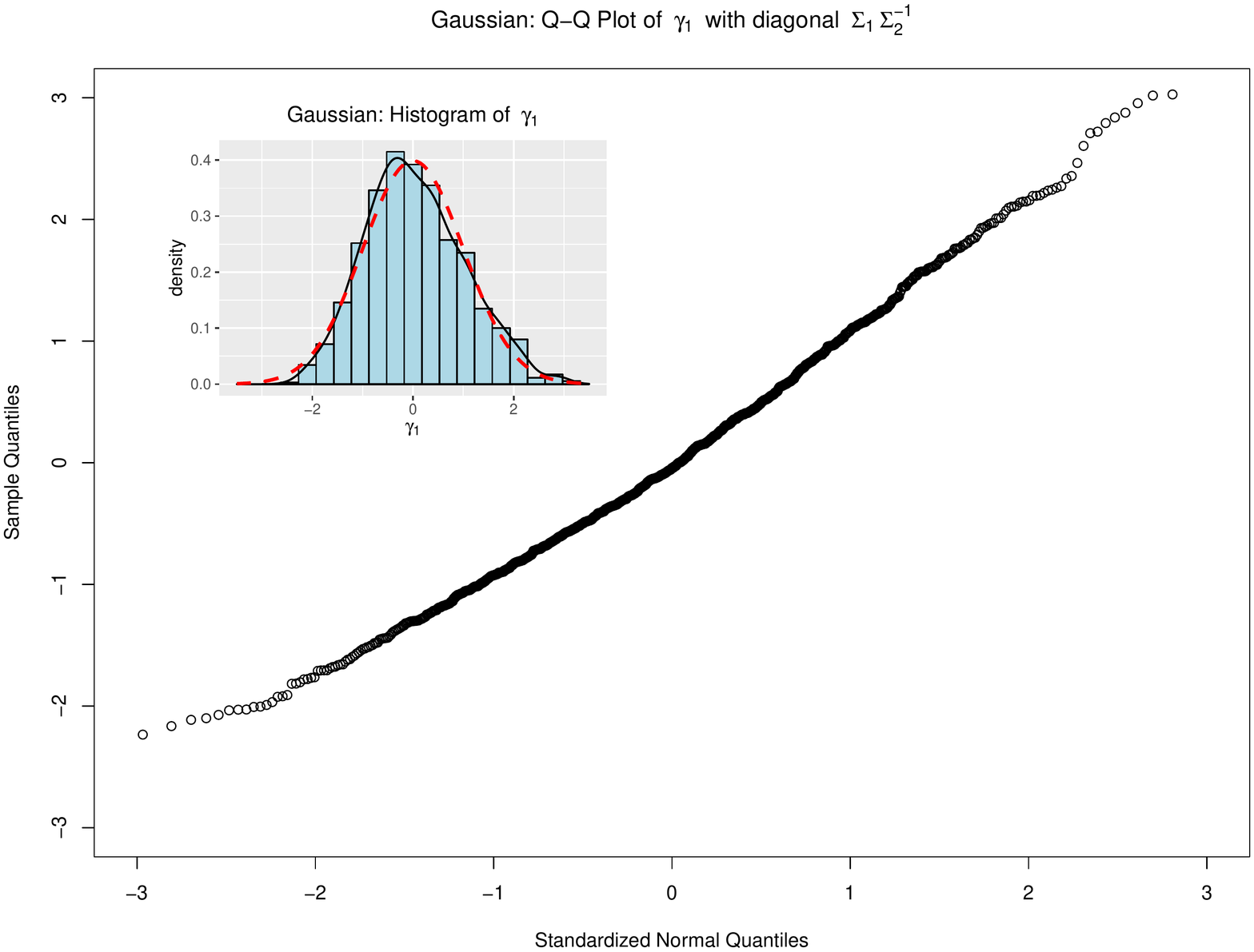}\quad\quad 
\includegraphics[width = .37\textwidth]{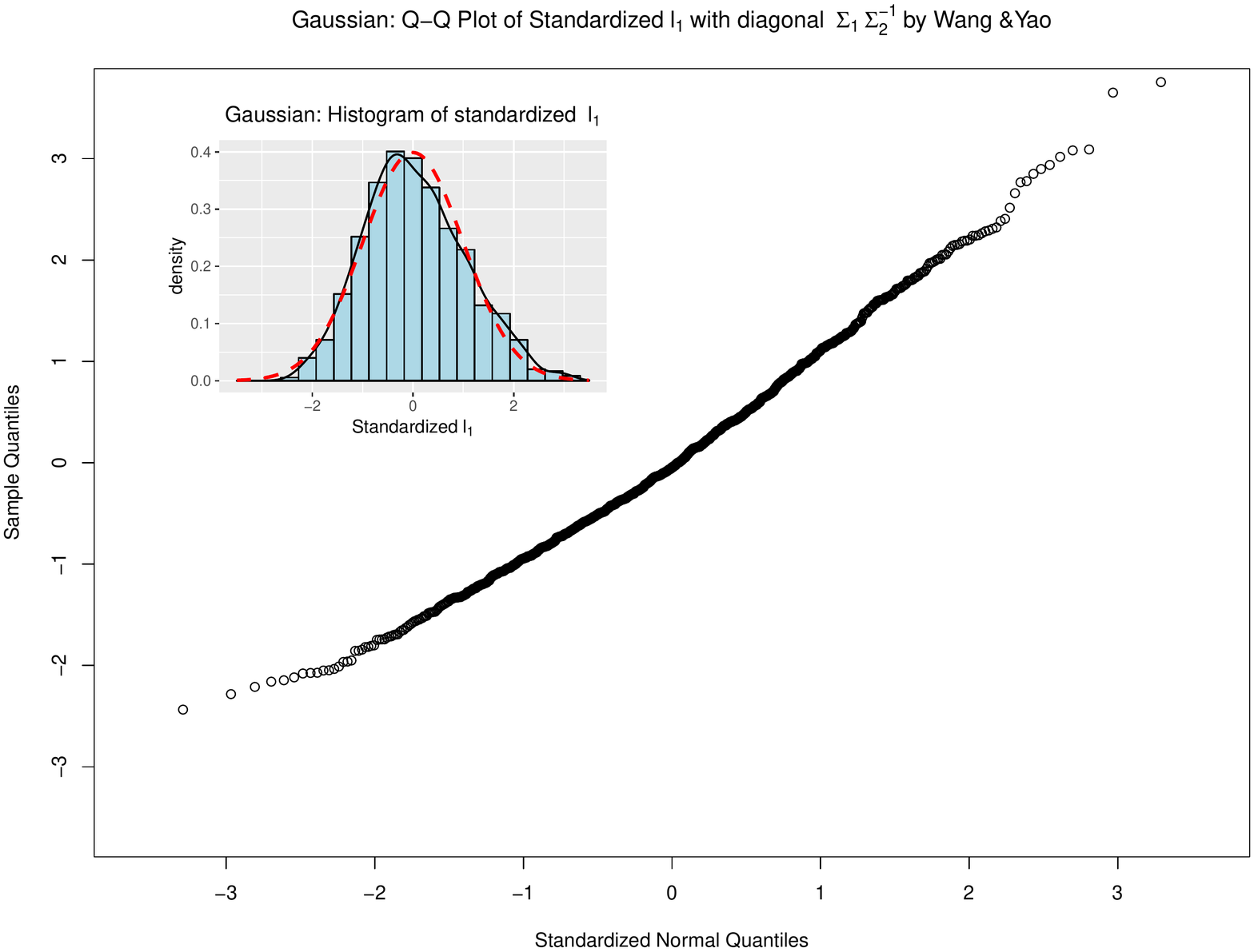}\\
\includegraphics[width = .37\textwidth]{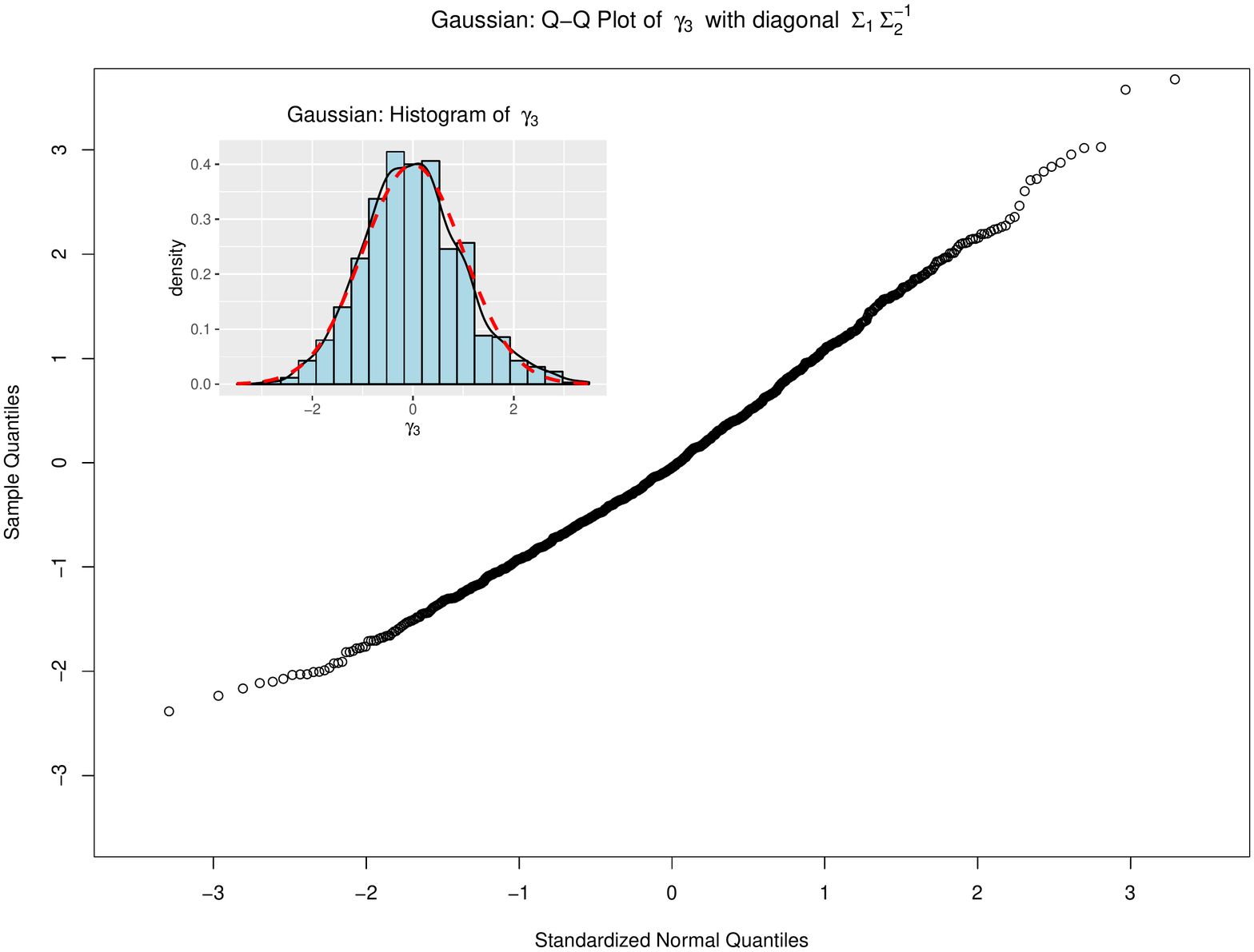}\quad\quad 
\includegraphics[width = .37\textwidth]{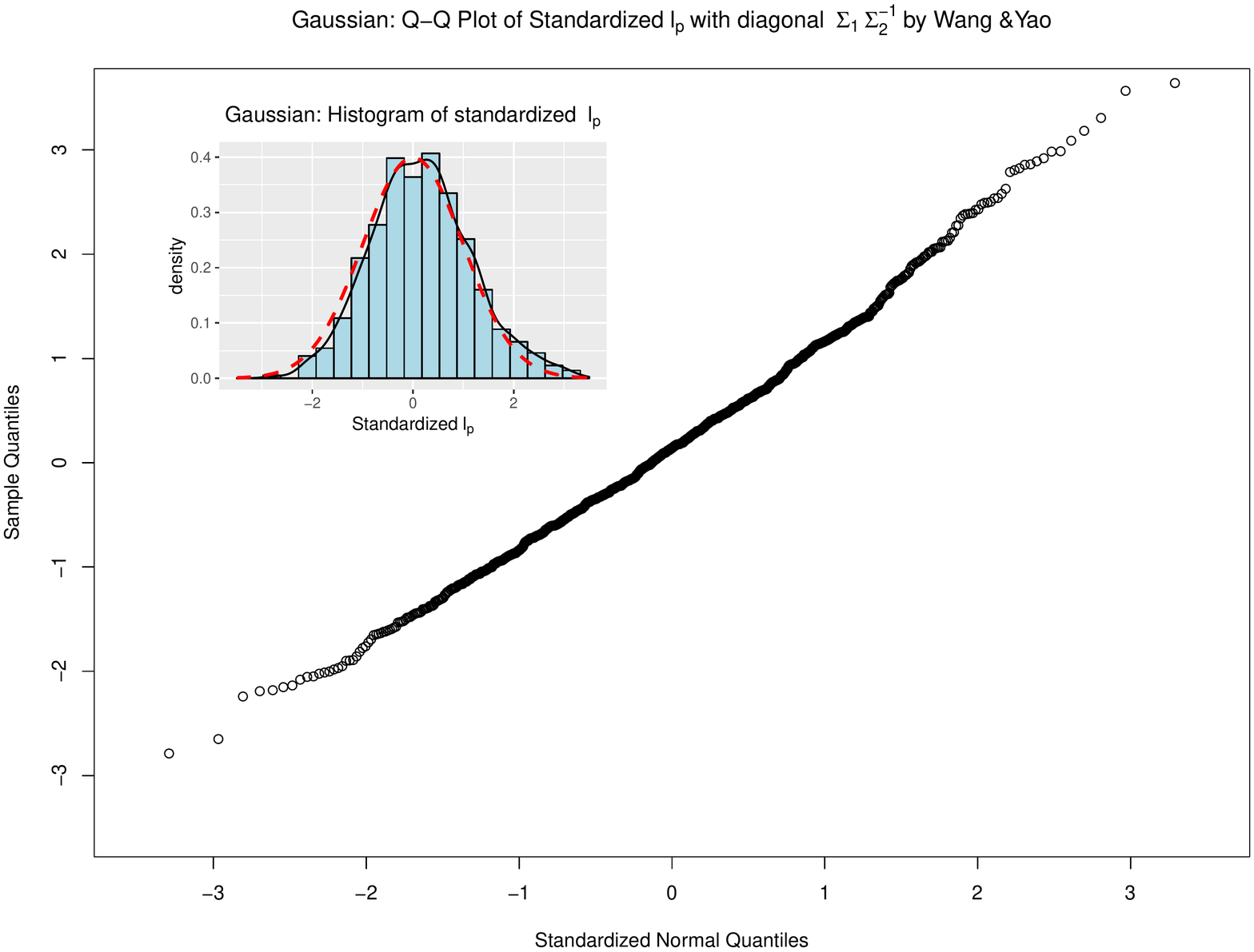}\\
\includegraphics[width = .3\textwidth]{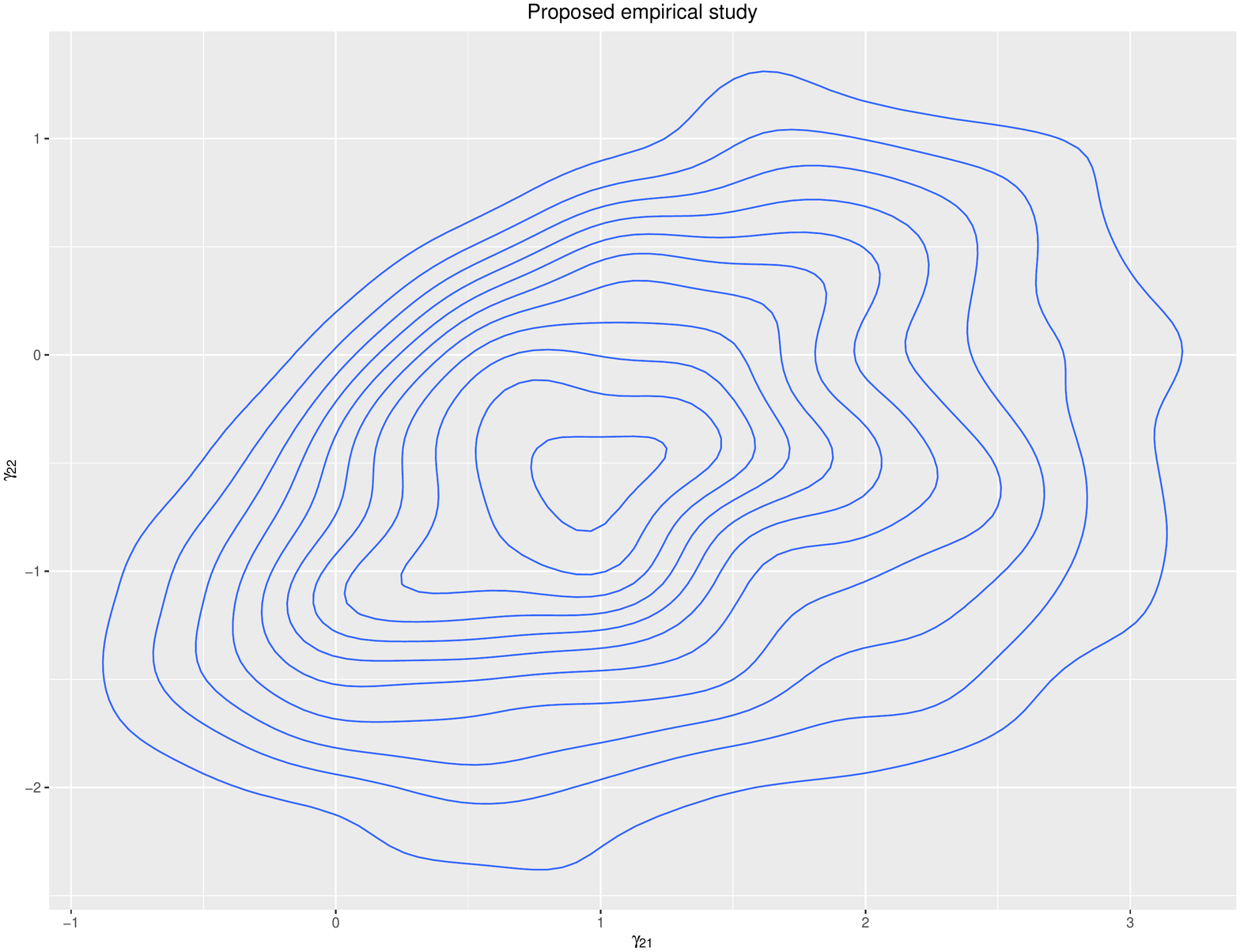}\quad \includegraphics[width = .3\textwidth] {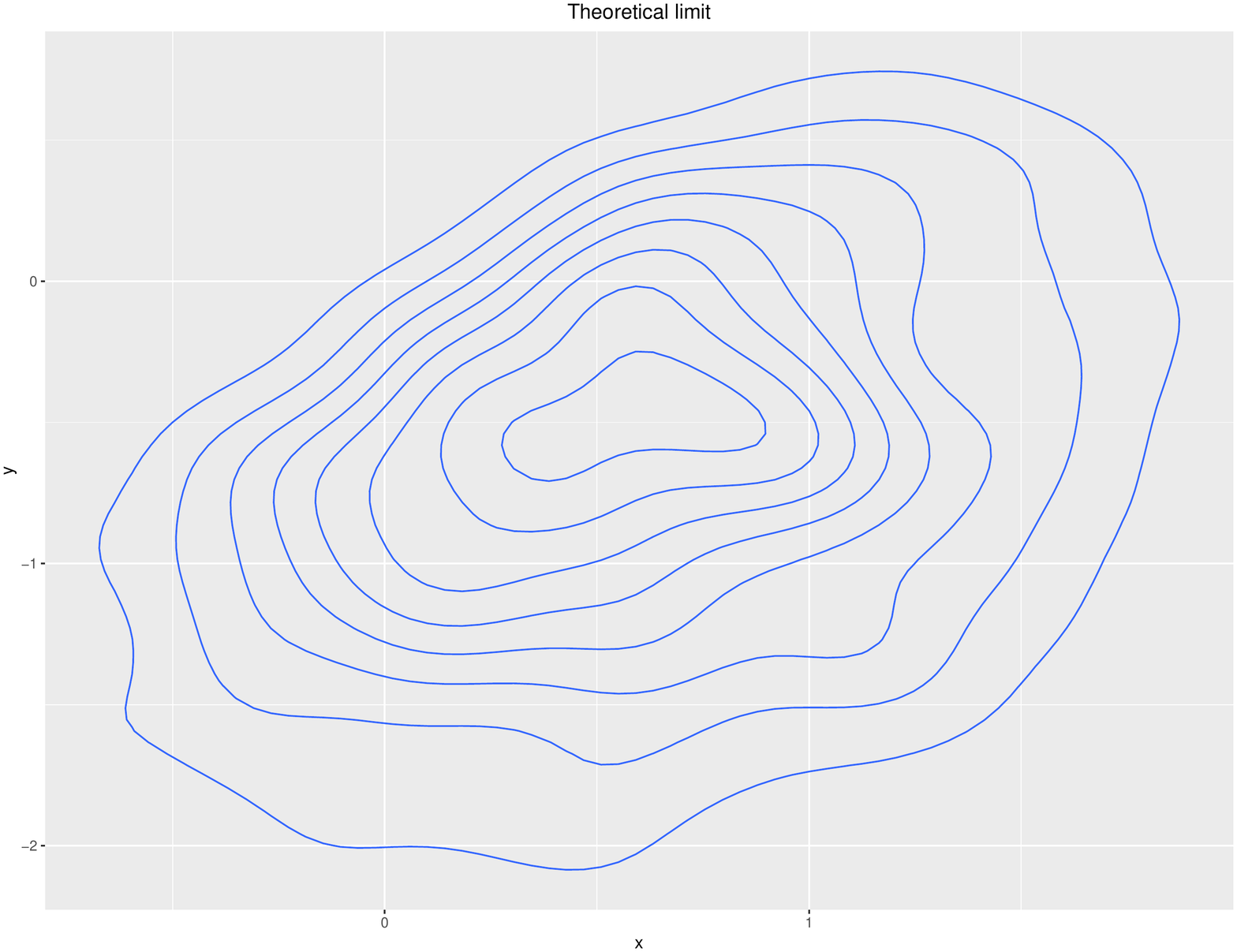}\quad \includegraphics[width = .3\textwidth] {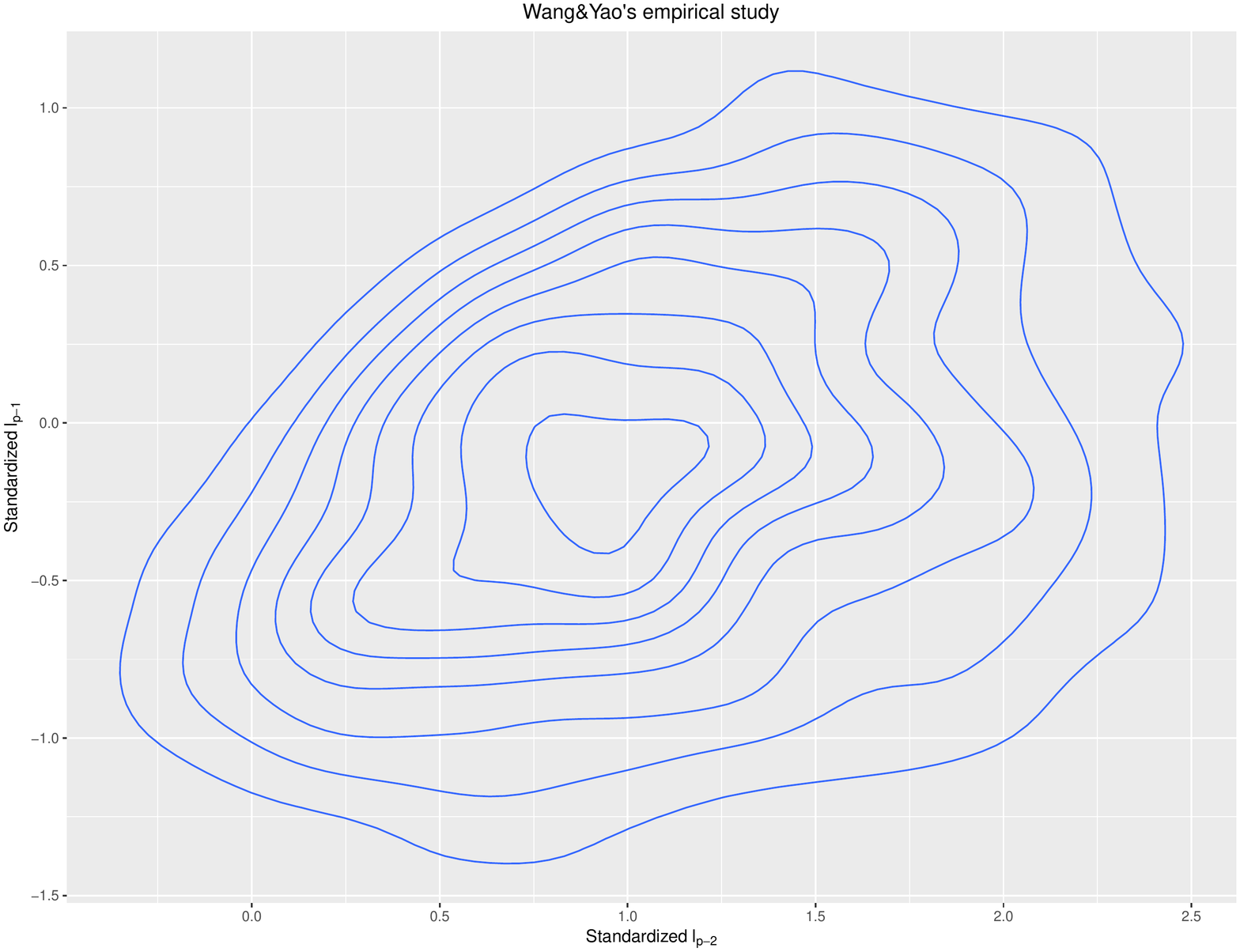}\\
\caption{ Case I under Gaussian assumption. Upper panels show that the  Q-Q plots for the proposed $\gamma_1$ and $\gamma_p$, as well as  the empirical densities of $\gamma_1$ and $\gamma_p$  (solid lines) comparing to their Gaussian limits (dashed lines). Lower panels show three  contour plots: the first is the proposed empirical joint density function of $(\gamma_{21}, \gamma_{22})$; the second is their corresponding limits; the third is the  empirical joint density function of
standardized $l_{p-2}$ and $l_{p-1}$.
 }\label{fig:1}
\end{center} 
\end{figure}

 \subsection{Case I under Binomial  Assumption}
 Continue to use the assumptions in {\rm \bf Case~I}, but   $x_{ij}$ and $y_{ij}$ are from Binomial  assumption. 
 Then by the Theorem~\ref{rmk2}, similarly we have 
\begin{itemize}
\item First, for  the largest  population  spikes $\alpha_1=20$ and sample eigenvalue $l_{1,p}$, 
\[\gamma_{1}= \sqrt{p-4}\Big(\frac{l_{p,1}(\bF)}{\psi_{n,1}}-1\Big) \rightarrow N(0, \sigma_1^2)\]
where 
\[\psi_{n,1}=42.667;\quad  \sigma_1^2=1.116 .\]

For
the least  population  spikes $\alpha_p=0.1$  and sample eigenvalues   $l_{p,p}$, we have 
\[\gamma_3= \sqrt{p-4}\Big(\frac{l_{p,p}(\bF)}{\psi_{n,3}}-1\Big) \rightarrow N(0, \sigma_3^2)\]
where 
\[\psi_{n,3}= 0.0737;\quad  \sigma_3^2=0.180.\] 

\item Second, for  the population  spikes $\alpha_2=0.2$ with multiplicity 2, and the sample eigenvalue $l_{1,p-1}$ and
 $l_{1,p-2}$, it is obtained that the two-dimensional random vector 
\[\gamma_2=\left(\gamma_{2,1}, \gamma_{2,2}\right)'=\left( \sqrt{p-4}\Big(\frac{l_{p,p-2}(\bF)}{\psi_{n,2}}-1\Big),
\sqrt{p-4}\Big(\frac{l_{p,p-1}(\bF)}{\psi_{n,2}}-1\Big)\right)' \]
converges to the eigenvalues of random matrix $-\frac1{\kappa_s}\left[\Omega_{\psi_{2}}\right]_{22}$,
where  $\psi_{n,2}=0.13$, $\kappa_s=1.433$ for the spike $\alpha_2=0.2$.  Furthermore, the
matrix 
$\left[\Omega_{\psi_{2}}\right]_{22}$
 is a $2\times 2$ symmetric matrix with the independent Gaussian entries, of which the $(s,t)$ element has mean zero and the variance given by 
 \[var(w_{st})=
\left\{
\begin{array}{cc}
0.264, &~ \text{if}~ s =t  \\
 1.160,  &  ~ \text{if}~ s \neq t  
\end{array}
\right.
\]
 \end{itemize}
Figure~2 depicts 
the simulated empirical distributions of the spiked eigenvalues  from binomial population under {\rm \bf Case~I} comparing with their  limiting distributions and the results from Wang and Yao (2017). 

 \begin{figure}[htbp]
\begin{center}
\includegraphics[width = .37\textwidth]{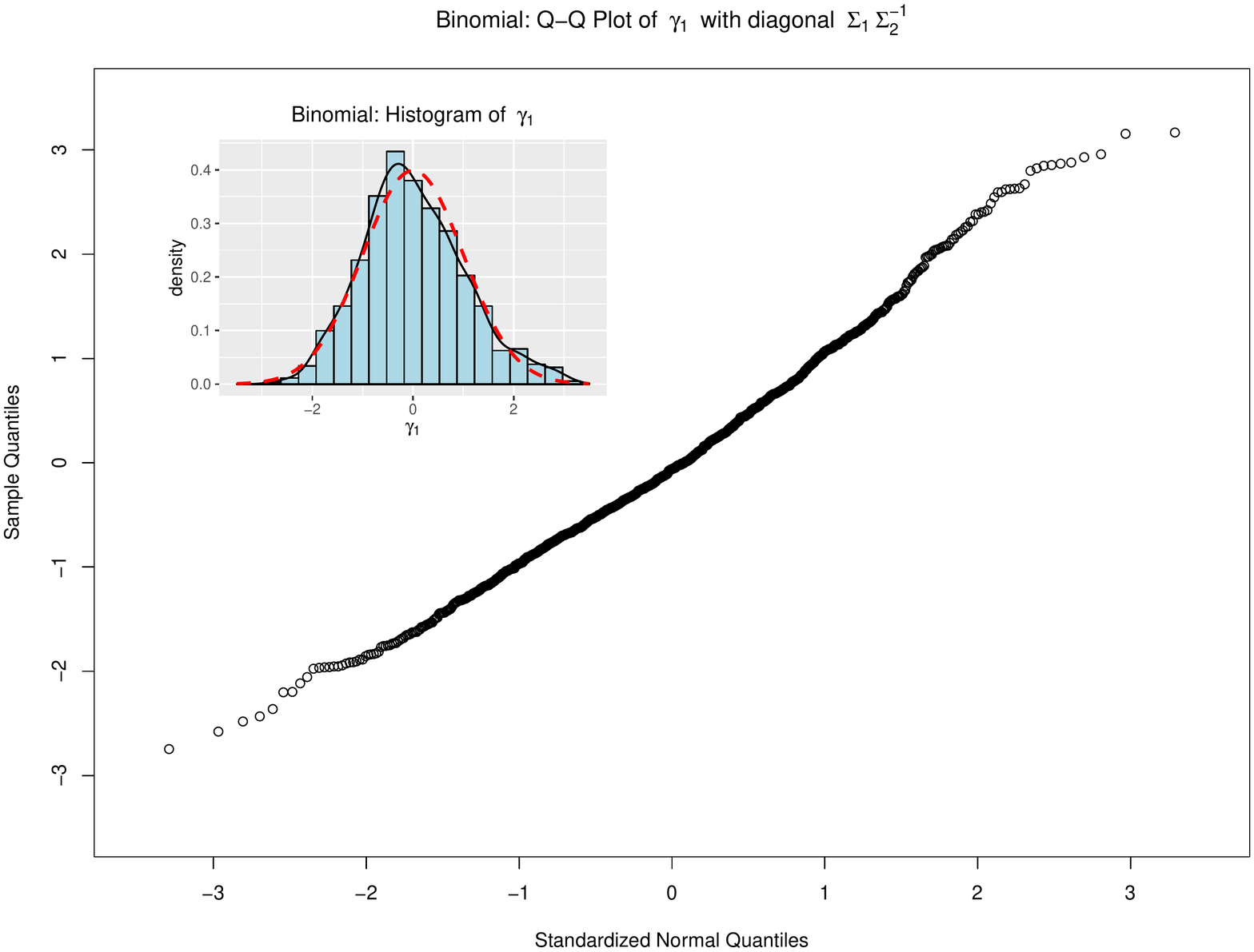}\quad\quad 
\includegraphics[width = .37\textwidth]{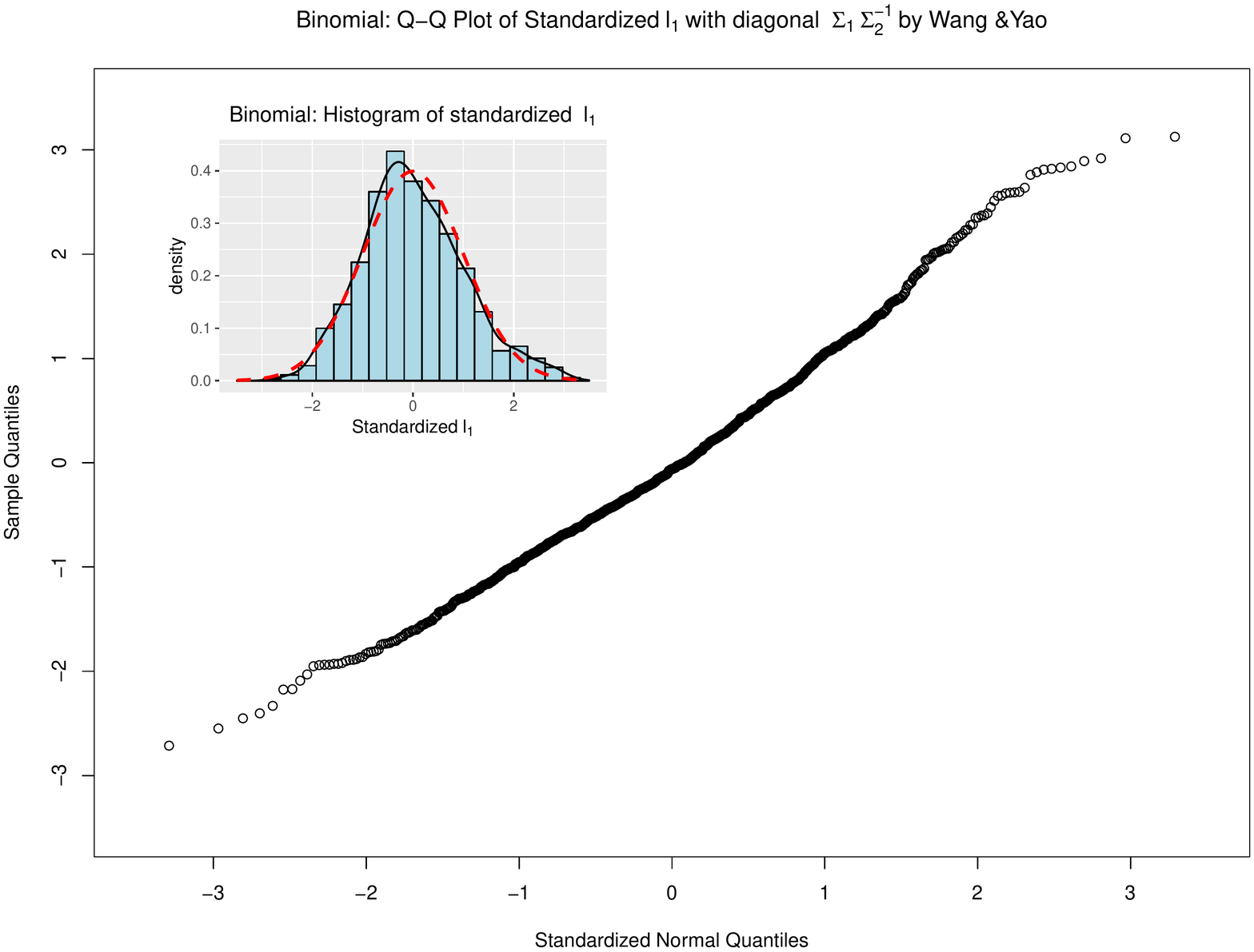}\\
\includegraphics[width = .37\textwidth]{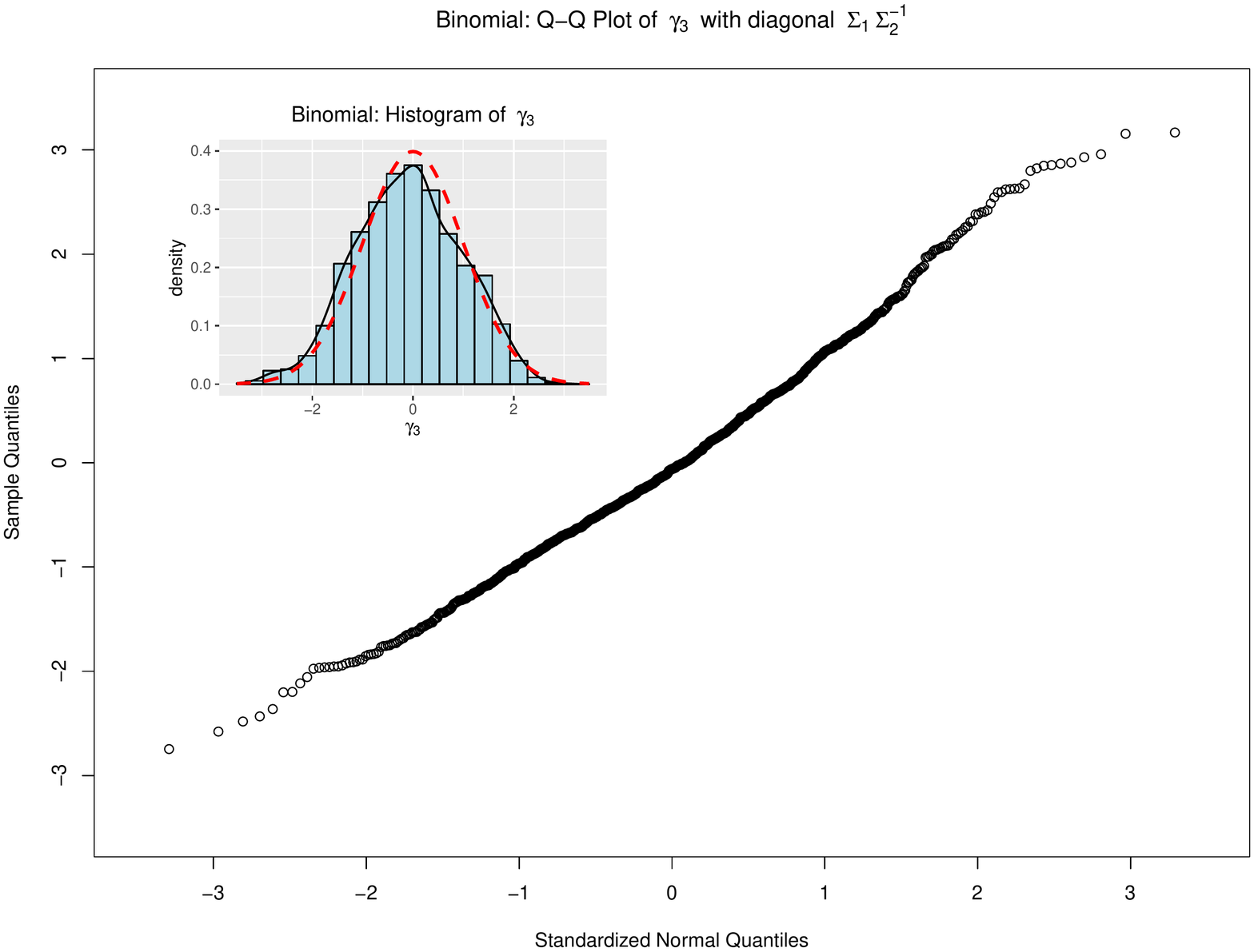}\quad\quad 
\includegraphics[width = .37\textwidth]{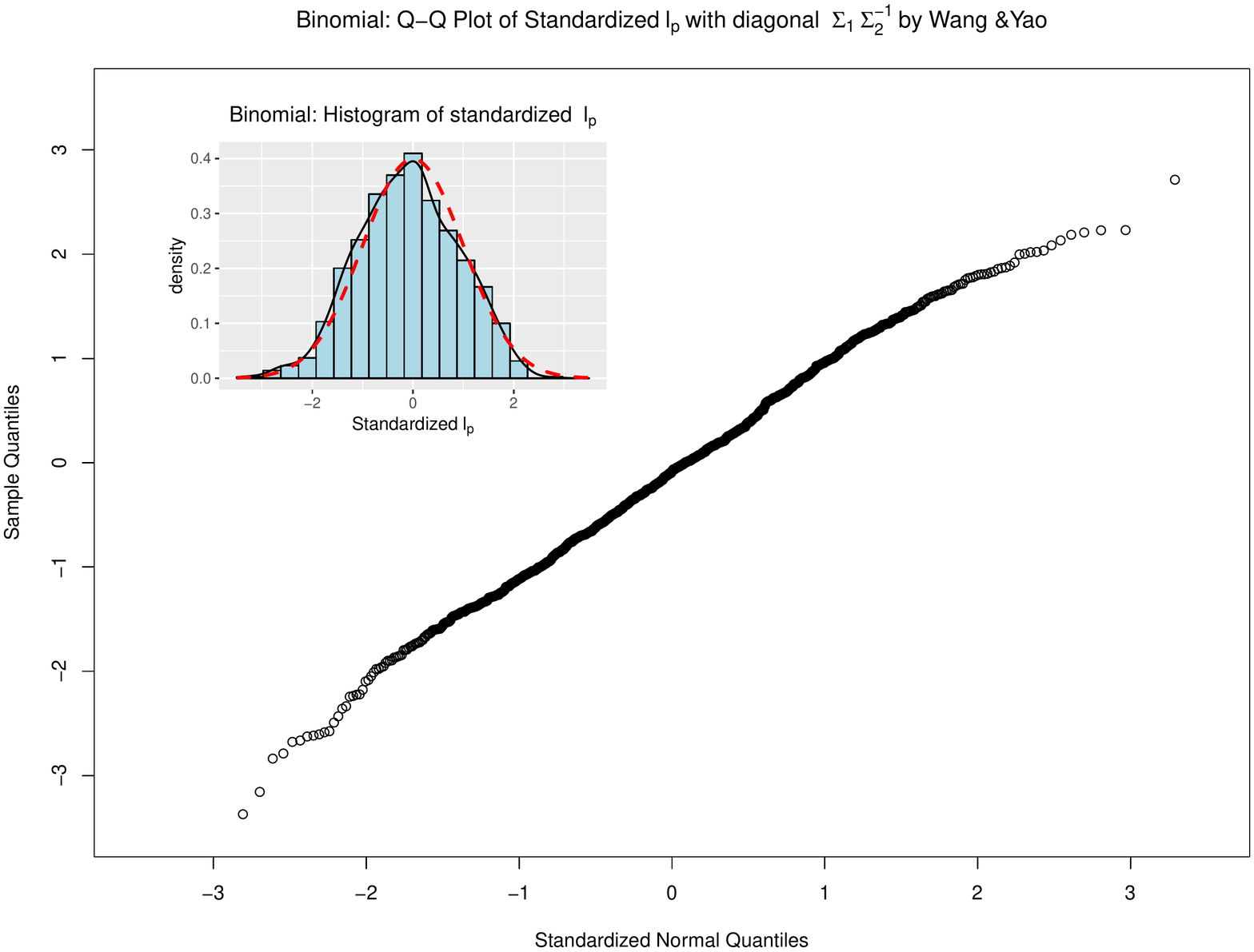}\\
\includegraphics[width = .3\textwidth]{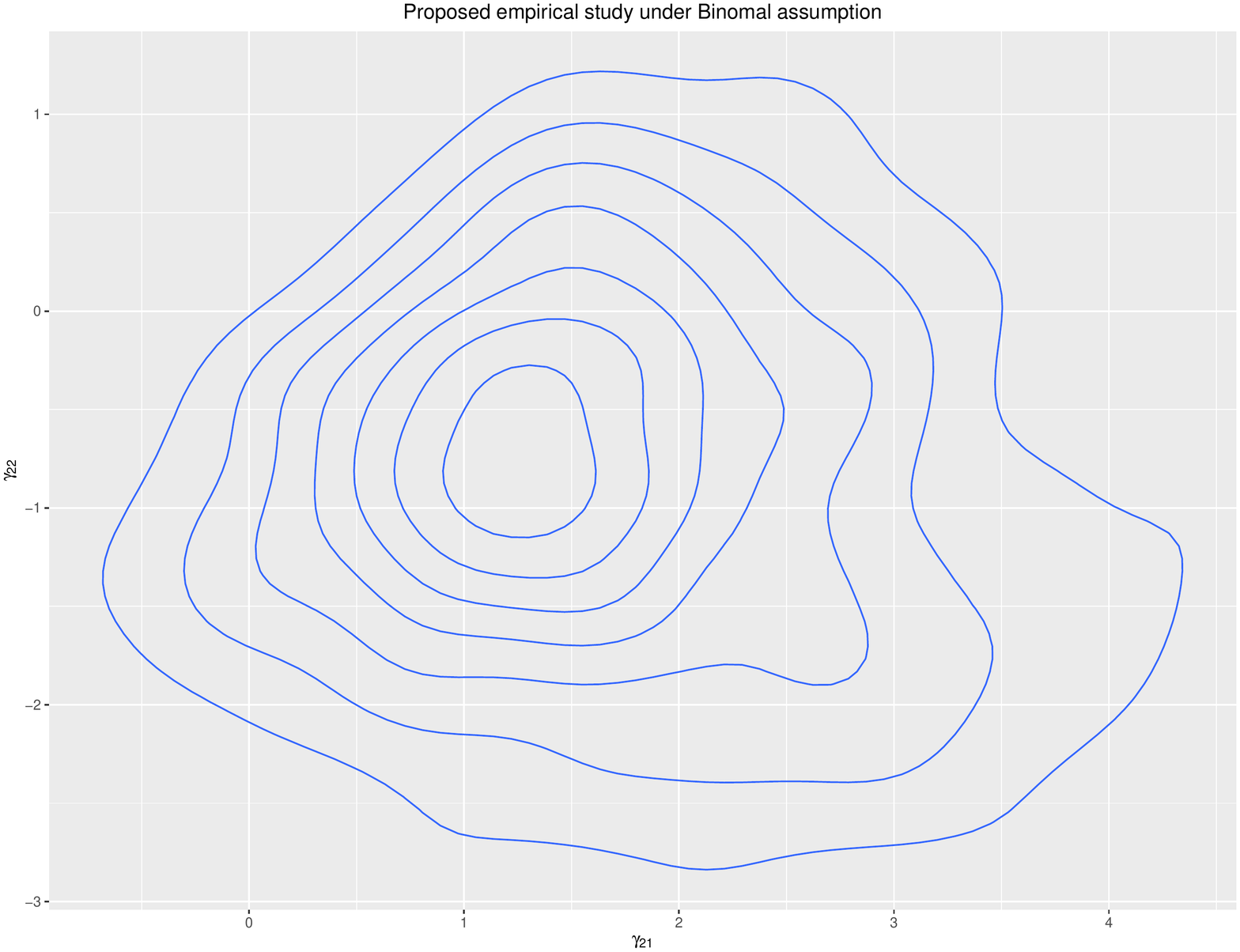}\quad \includegraphics[width = .3\textwidth] {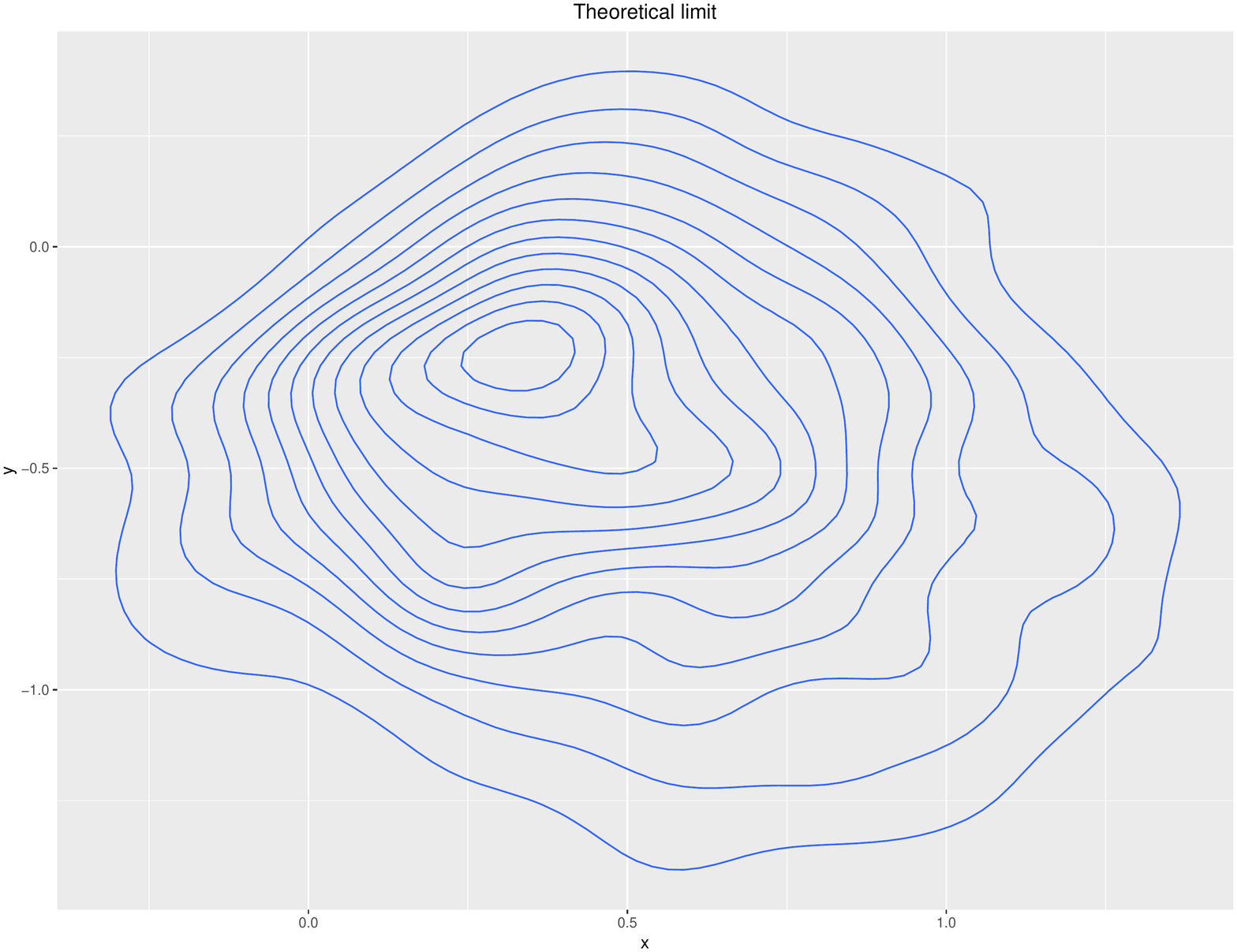}\quad \includegraphics[width = .3\textwidth] {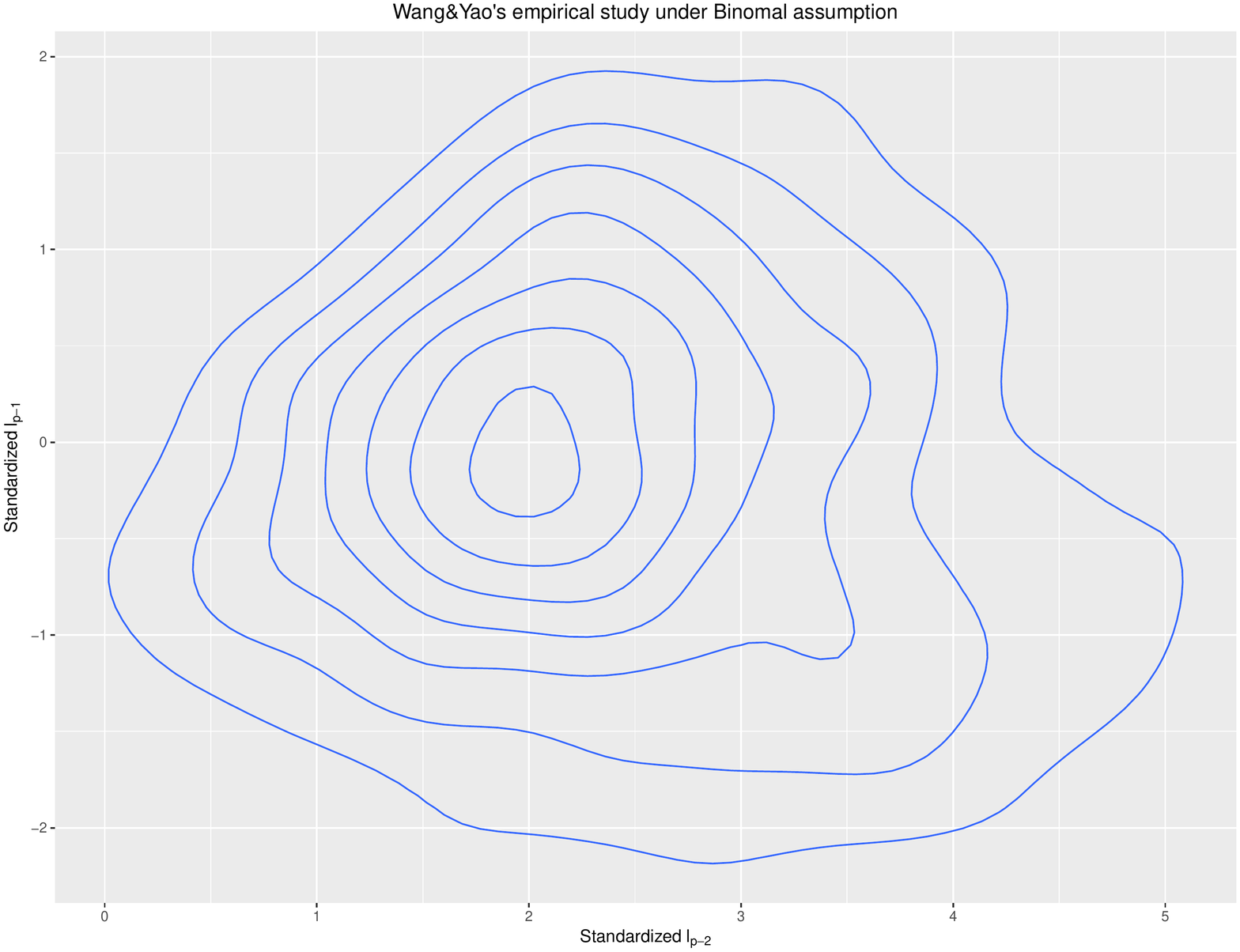}\\
\caption{ Case I under Binomial assumption. Upper panels show that the  Q-Q plots for the proposed $\gamma_1$ and $\gamma_p$, as well as  the empirical densities of $\gamma_1$ and $\gamma_p$  (solid lines) comparing to their Gaussian limits (dashed lines). Middle panels are the corresponding comparison of the empirical density of  standardized $l_{1}$ and $l_p$ in Wang and Yao (2017).  Lower panels show three  contour plots: the first is the proposed empirical joint density function of $(\gamma_{21}, \gamma_{22})$; the second is their corresponding limits; the third is the  empirical joint density function of
standardized $l_{p-2}$ and $l_{p-1}$.
 }\label{fig:2}
\end{center} 
\end{figure}

 \subsection{{\rm \bf Case~II} under all  Assumptions}

For the {\rm \bf Case~II}, the simulations show that our proposed results are the same to the one of Normal assumption under {\rm \bf Case~I}  for all the population distribution assumptions by Theorem~{\ref{CLT}. The simulated results of three assumptions under {\rm \bf Case~II} are respectively depicted in Figures~\ref{fig:4}-\ref{fig:5}.  

 \begin{figure}[htbp]
\begin{center}
\includegraphics[width = .37\textwidth]{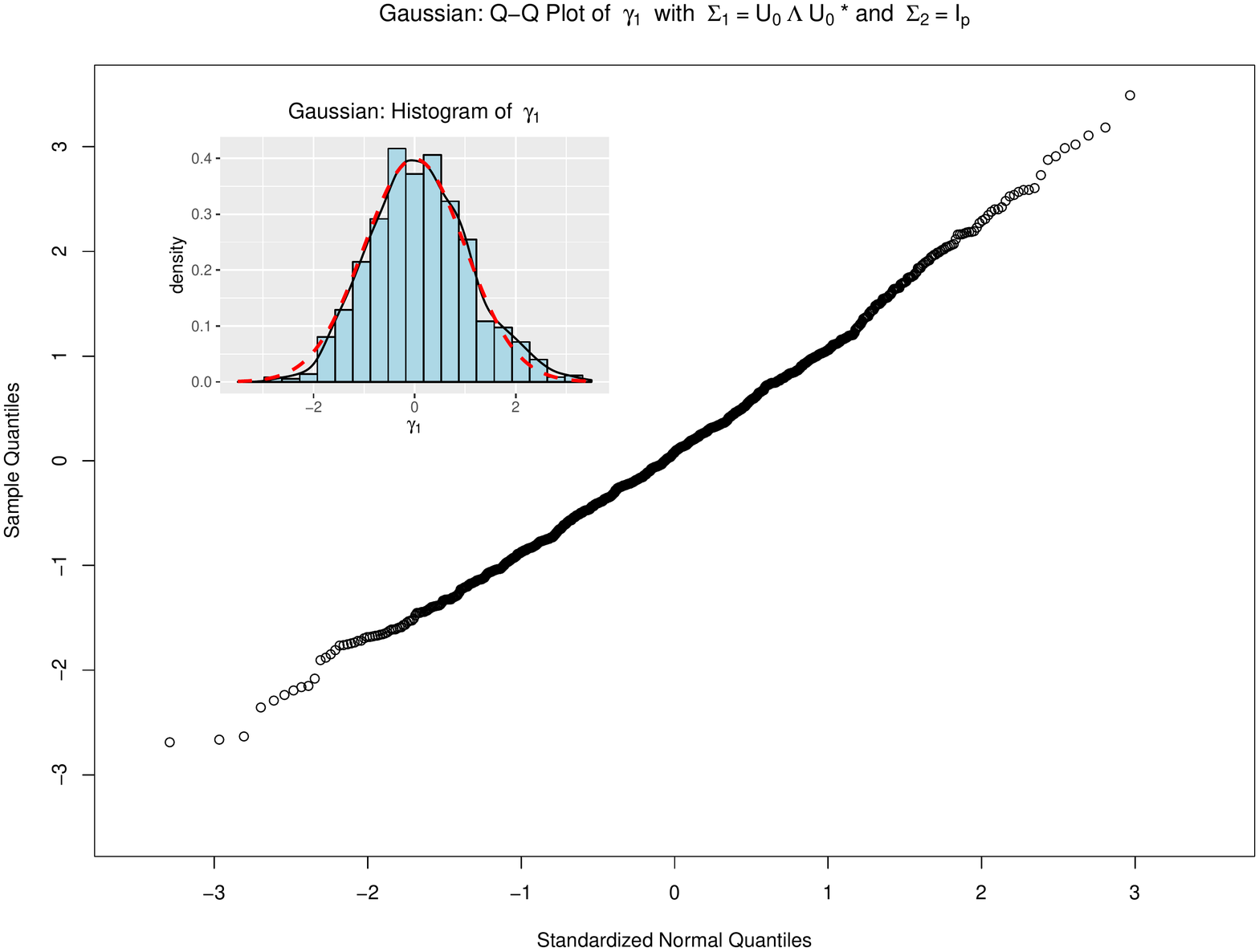}\quad\quad 
\includegraphics[width = .37\textwidth]{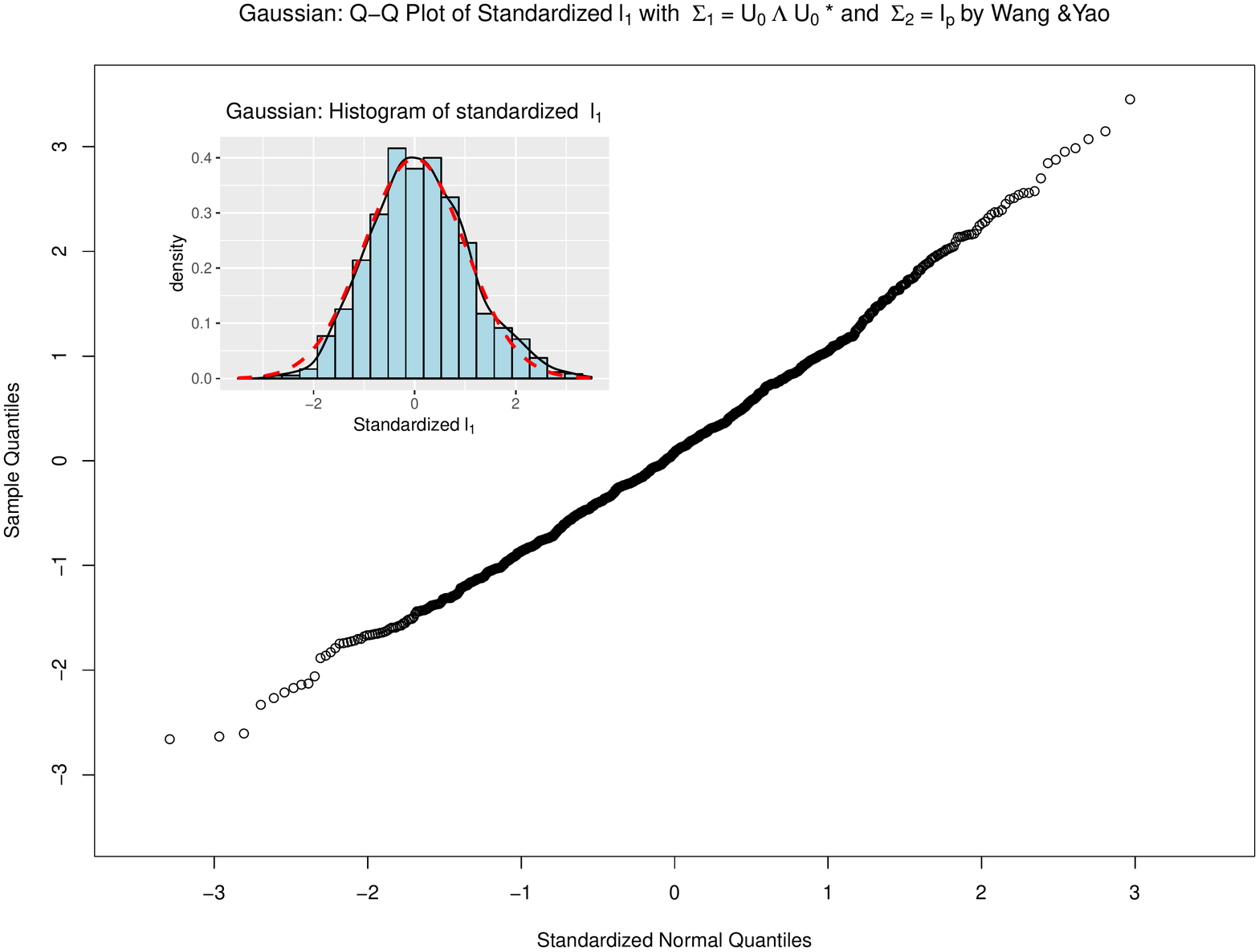}\\
\includegraphics[width = .37\textwidth]{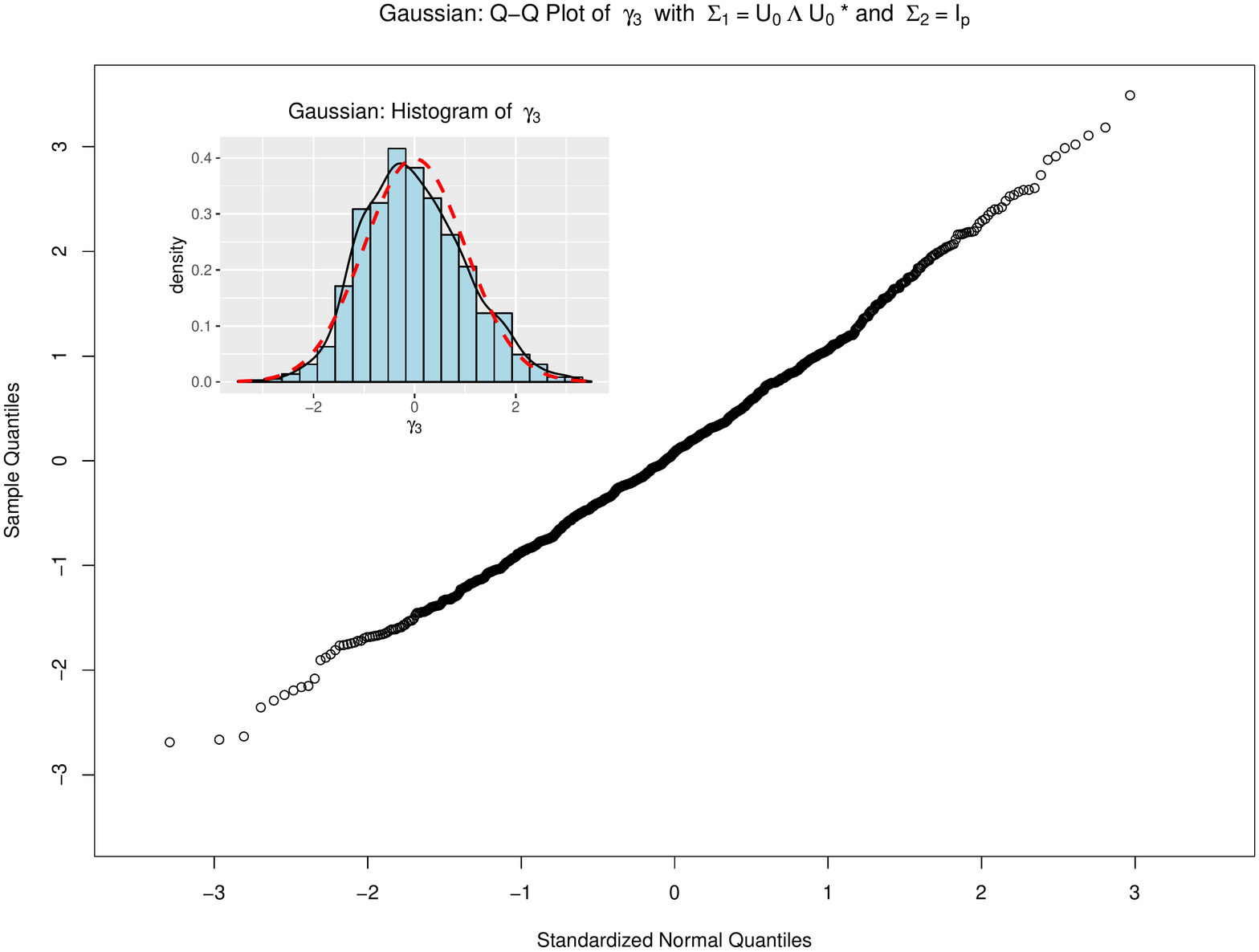}\quad\quad 
\includegraphics[width = .37\textwidth]{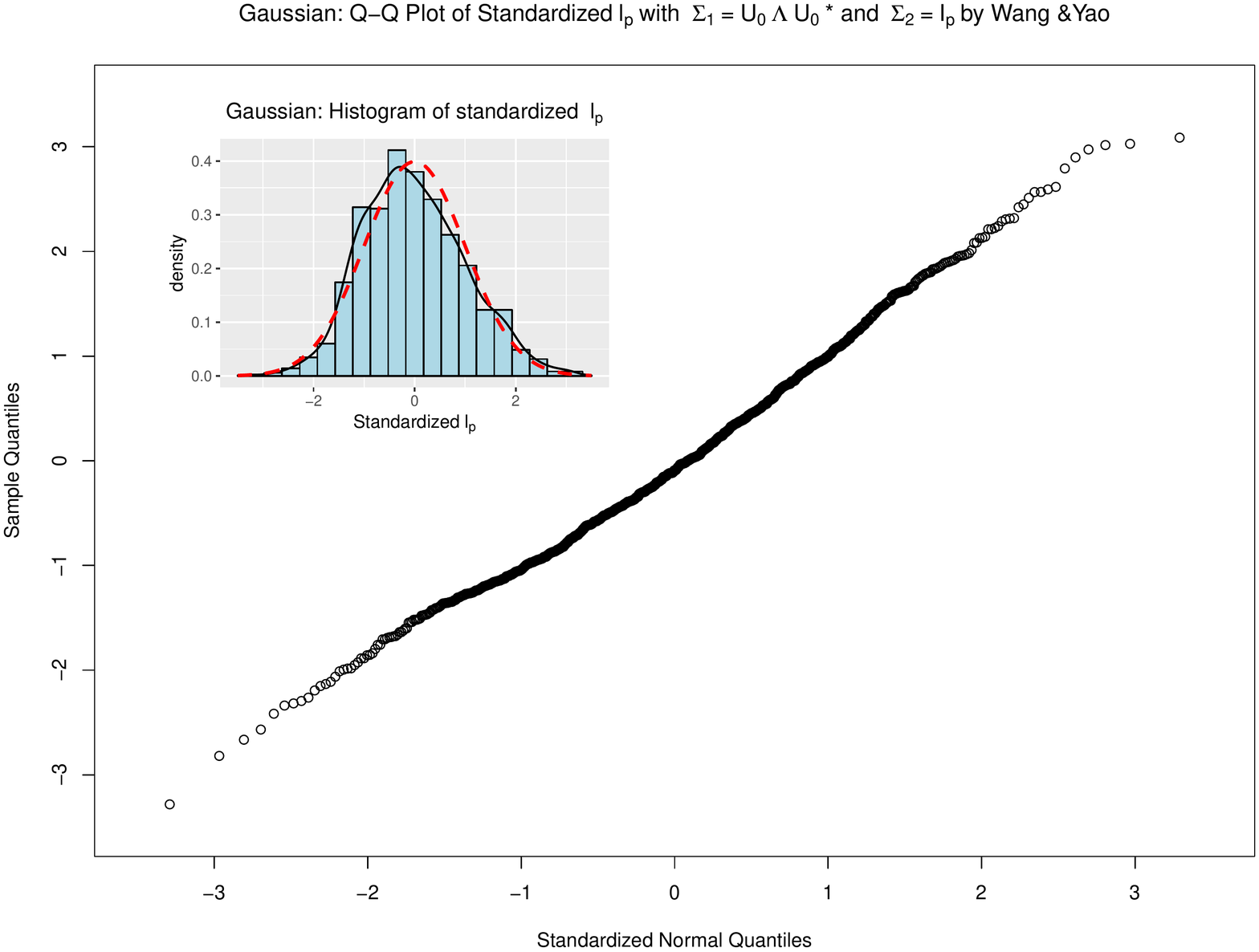}\\
\includegraphics[width = .3\textwidth]{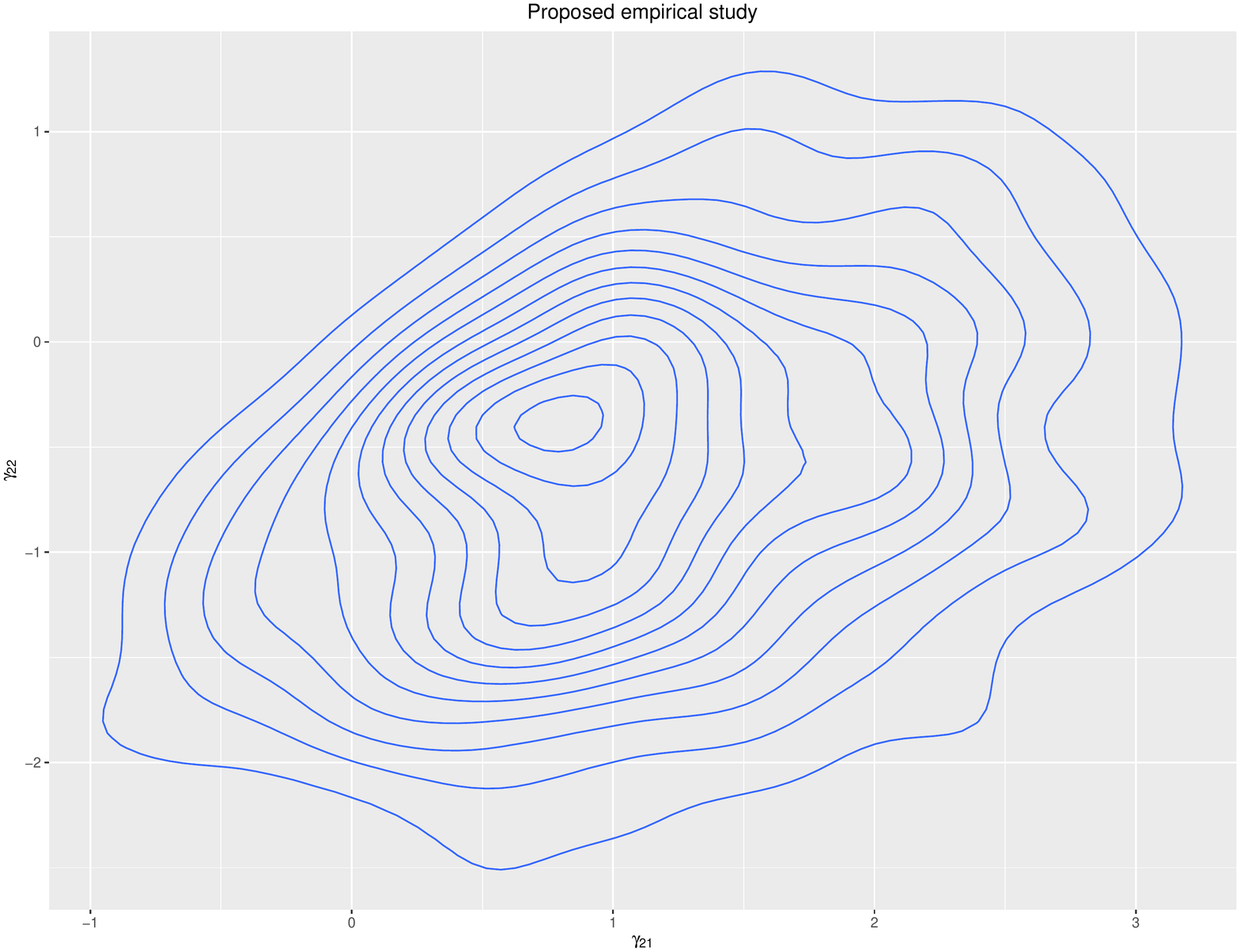}\quad \includegraphics[width = .3\textwidth] {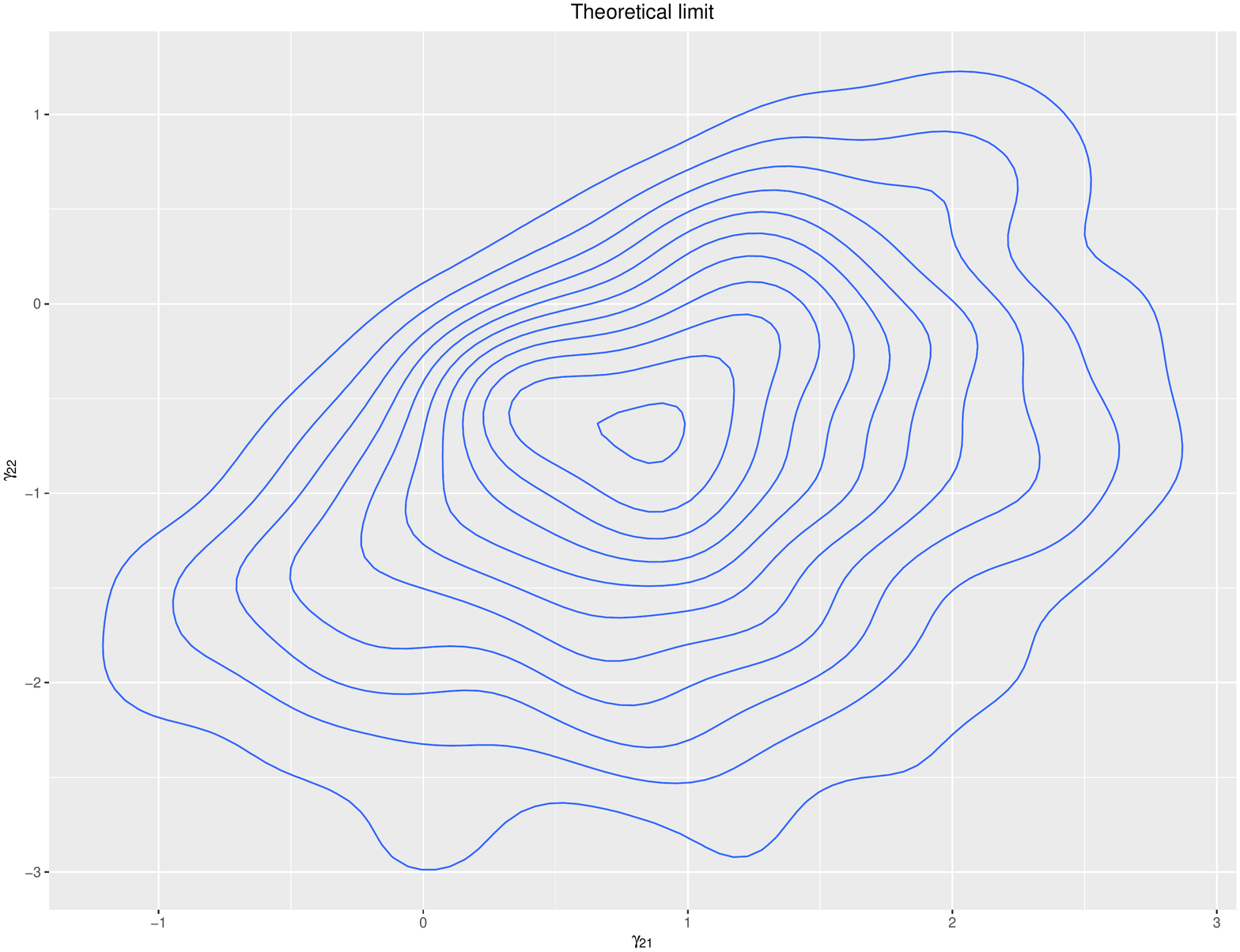}\quad \includegraphics[width = .3\textwidth] {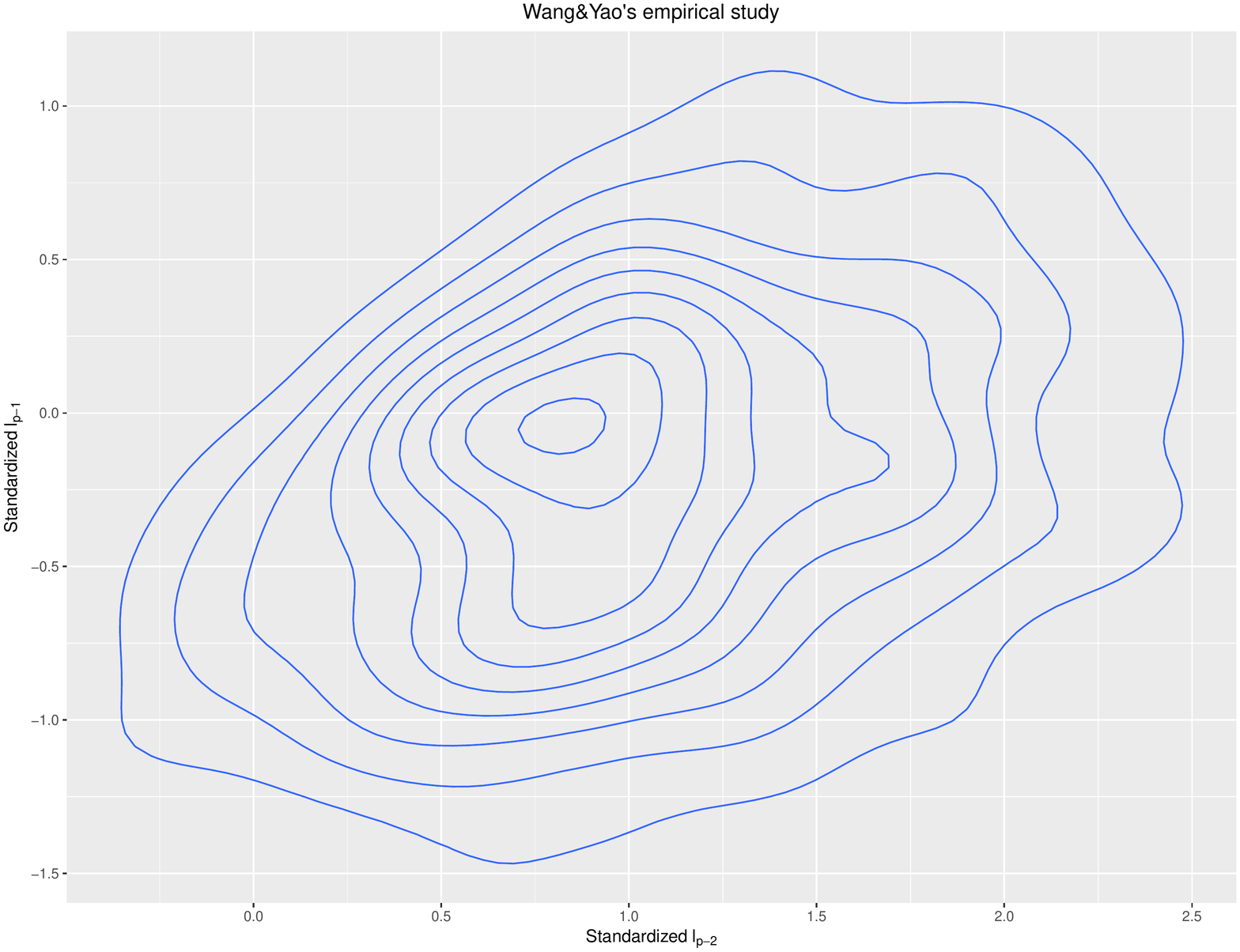}\\
\caption{ {\rm \bf Case~II} under Gaussian assumption. Upper panels show that the  Q-Q plots for the proposed $\gamma_1$ and $\gamma_3$, as well as  the empirical densities of $\gamma_1$ and $\gamma_p$  (solid lines) comparing to their Gaussian limits (dashed lines). Middle panels are the corresponding comparison of the empirical density of  standardized $l_{1}$ and $l_p$ in Wang and Yao (2017).  Lower panels show three  contour plots: the first is the proposed empirical joint density function of $(\gamma_{21}, \gamma_{22})$; the second is their corresponding limits; the third is the  empirical joint density function of
standardized $l_{p-2}$ and $l_{p-1}$.
 }\label{fig:4}
\end{center} 
\end{figure}

 \begin{figure}[htbp]
\begin{center}
\includegraphics[width = .37\textwidth]{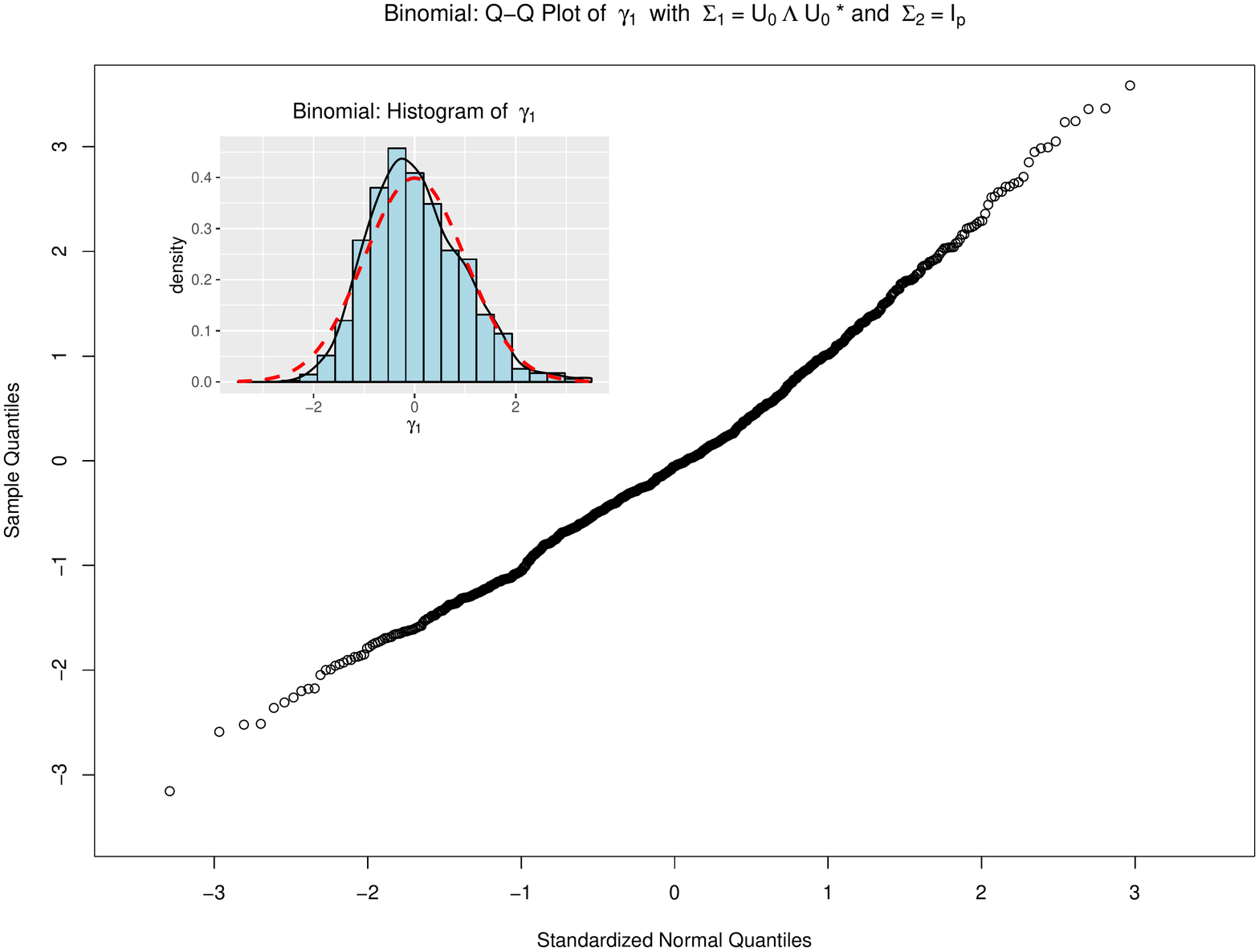}\quad\quad 
\includegraphics[width = .37\textwidth]{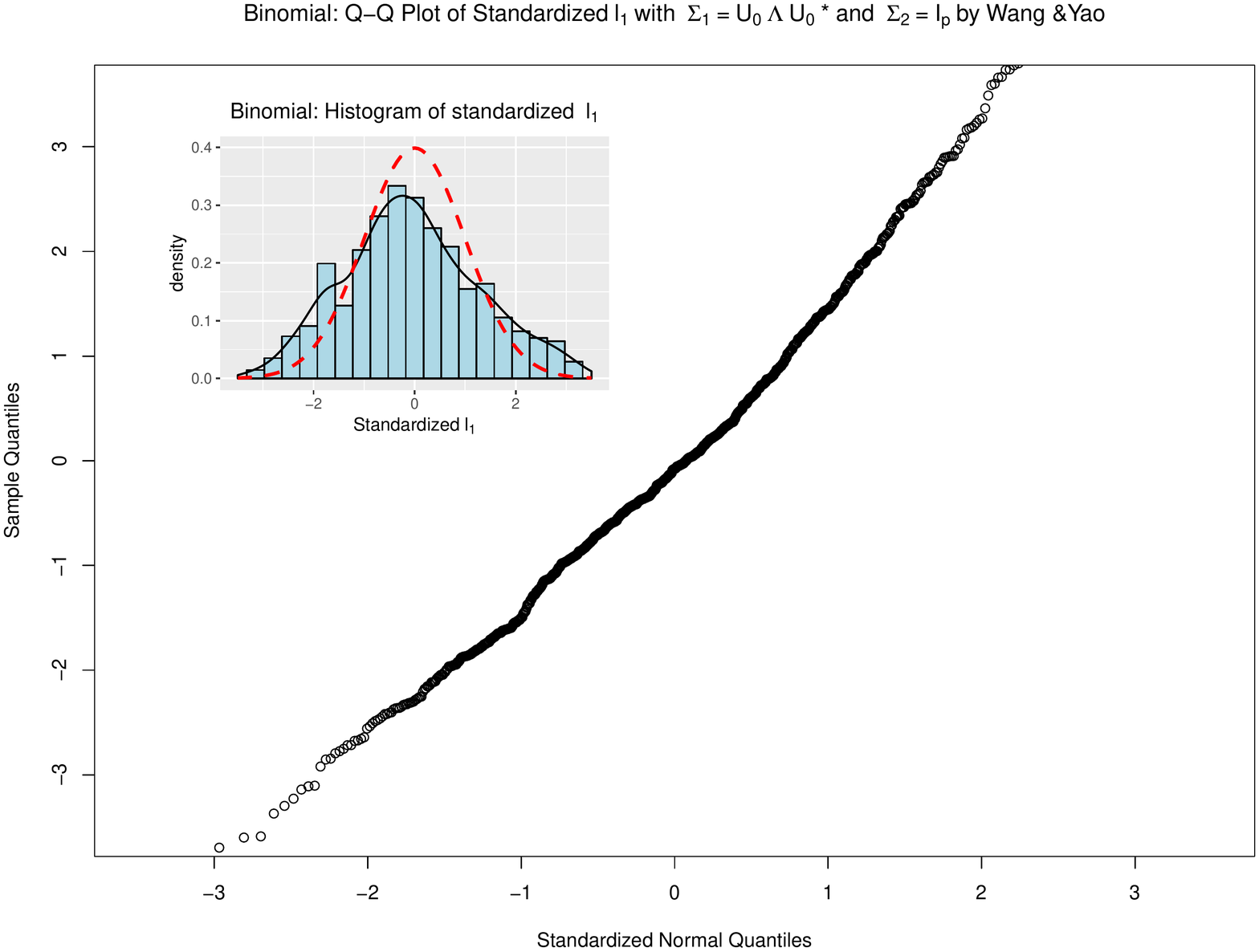}\\
\includegraphics[width = .37\textwidth]{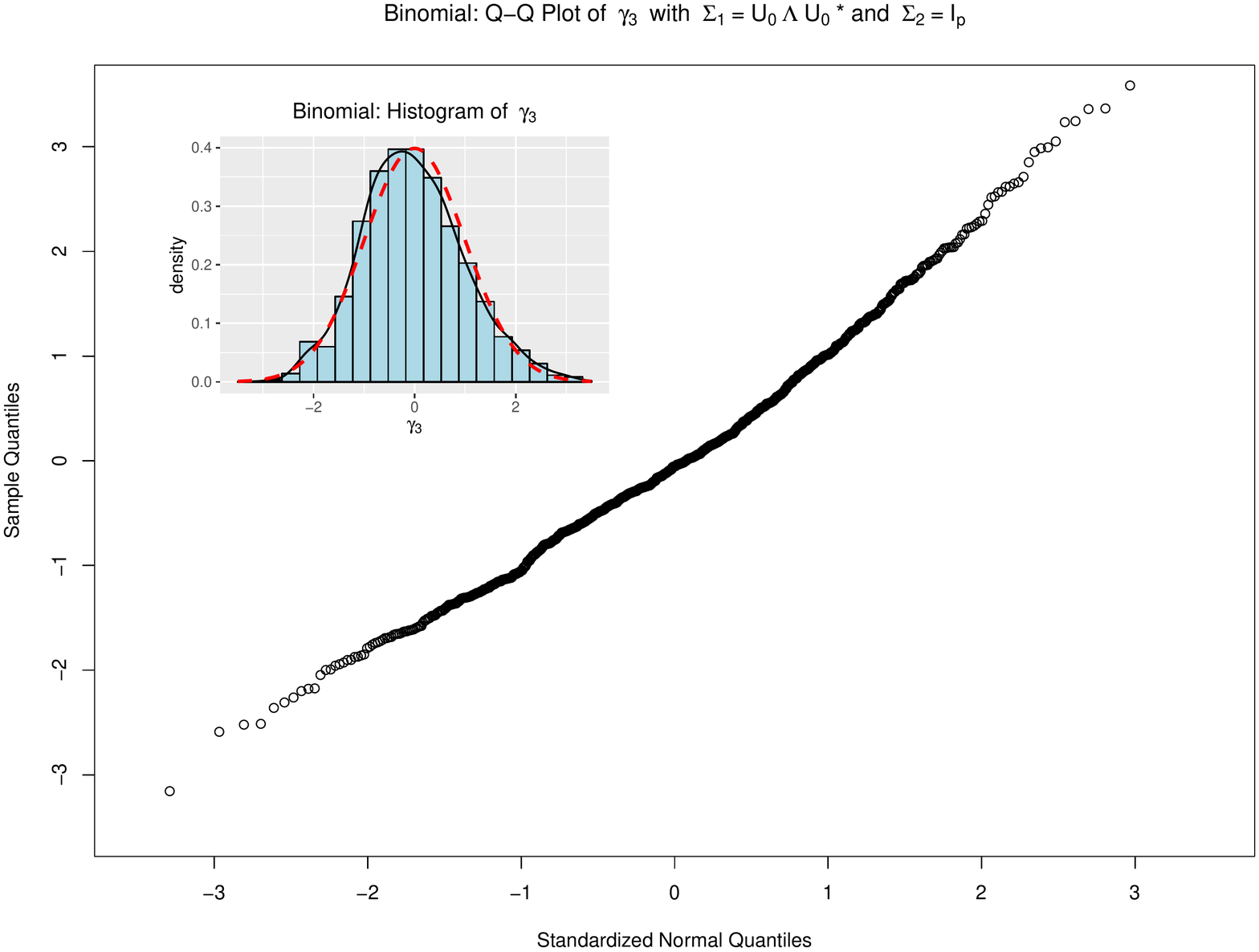}\quad\quad 
\includegraphics[width = .37\textwidth]{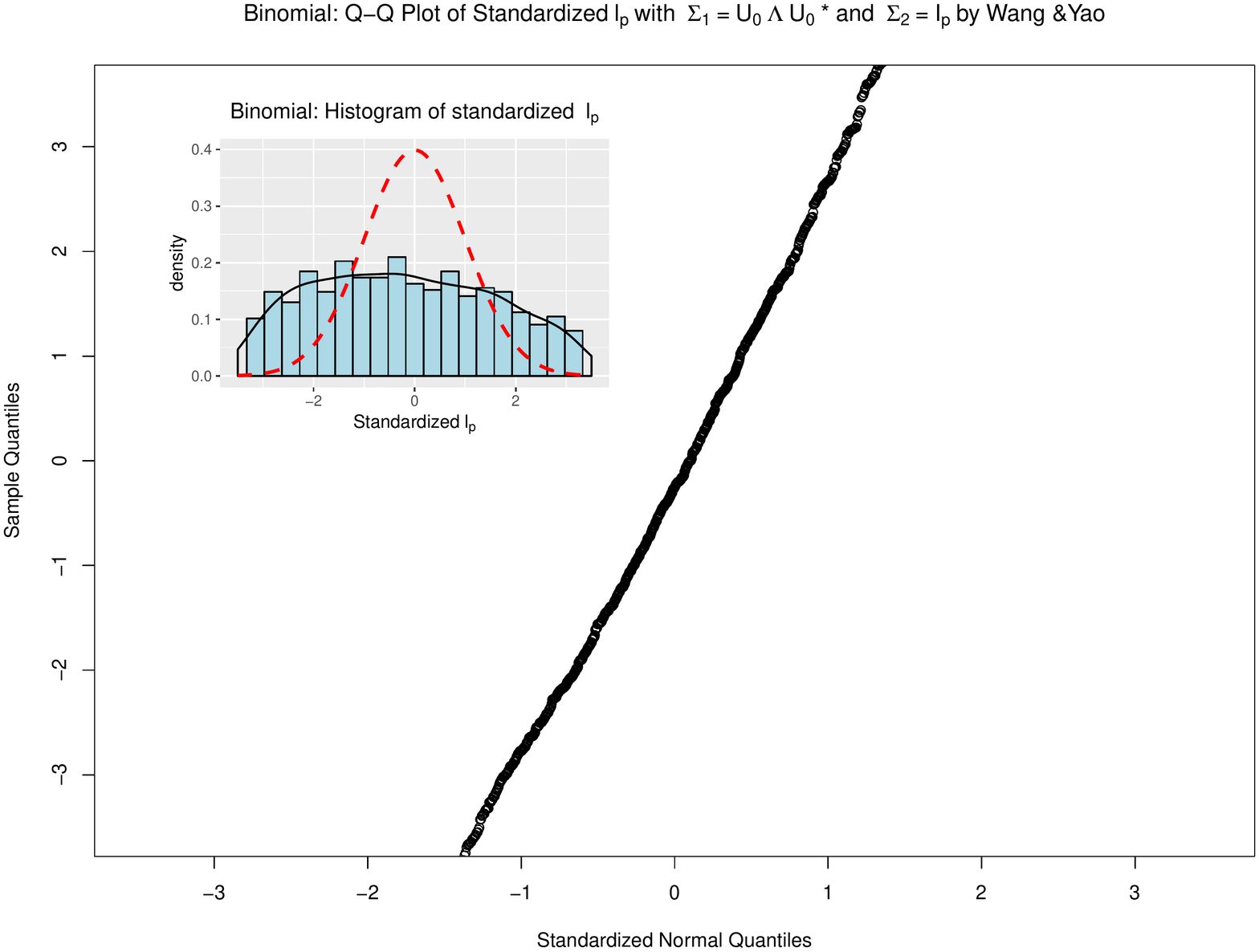}\\
\includegraphics[width = .3\textwidth]{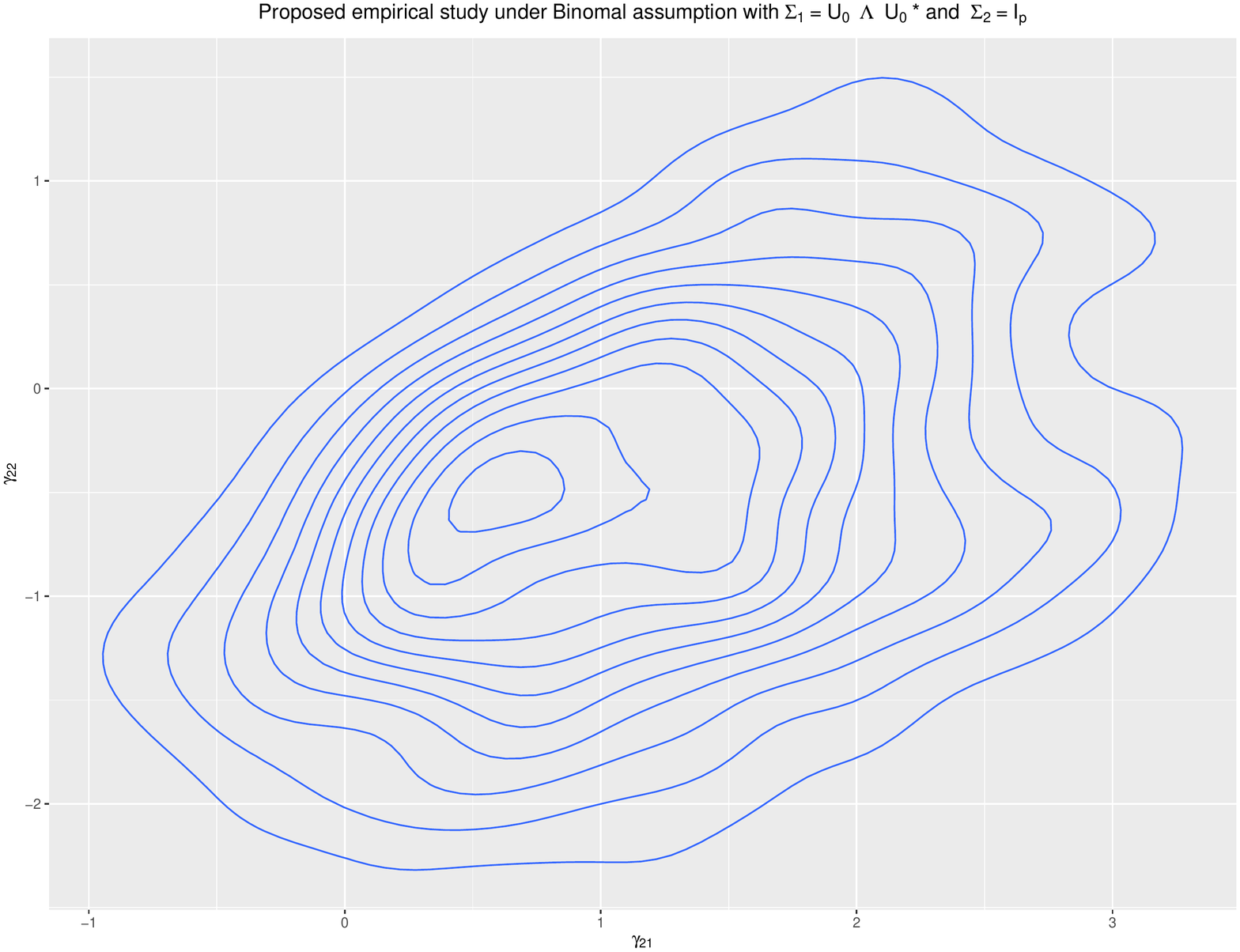}\quad \includegraphics[width = .3\textwidth] {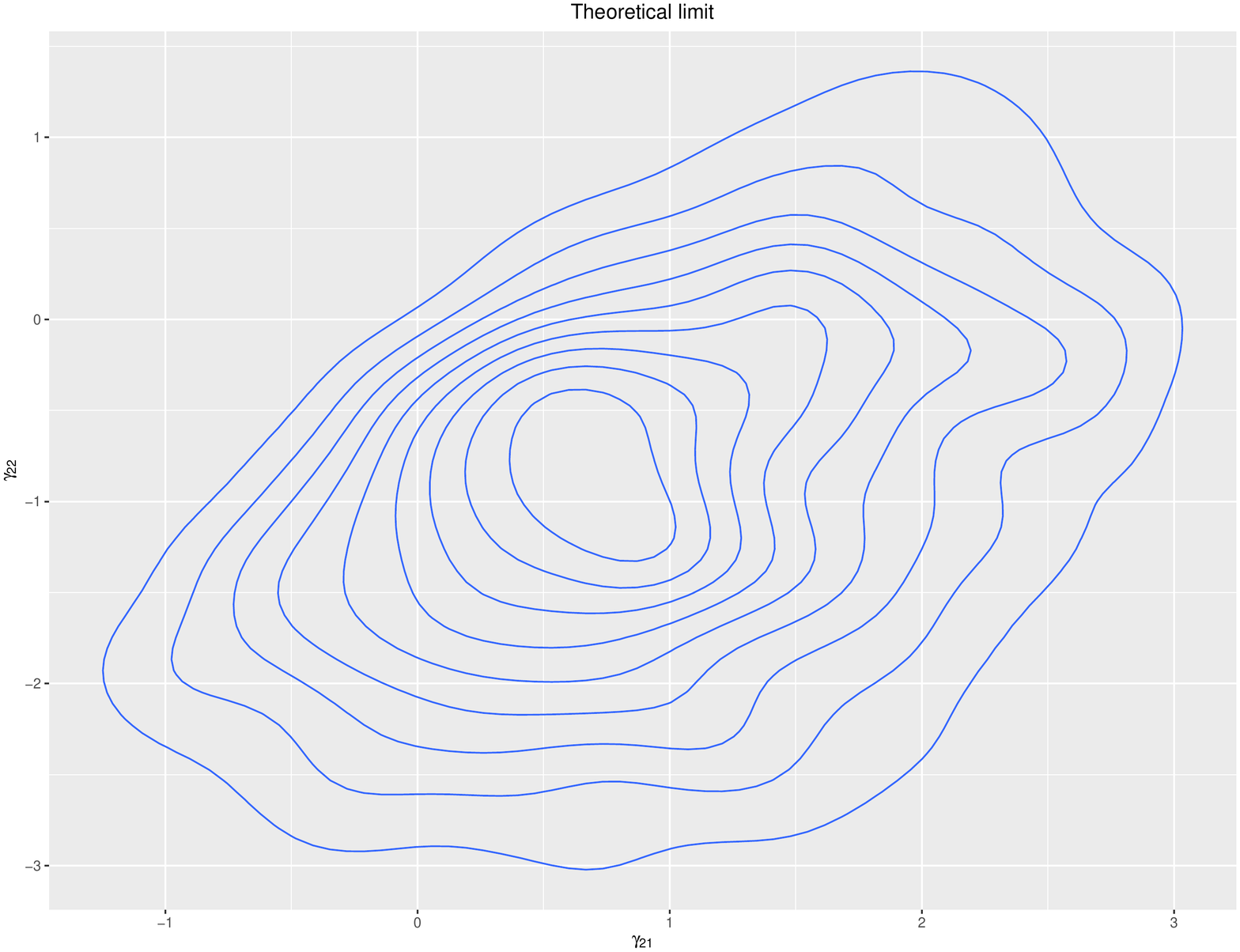}\quad \includegraphics[width = .3\textwidth] {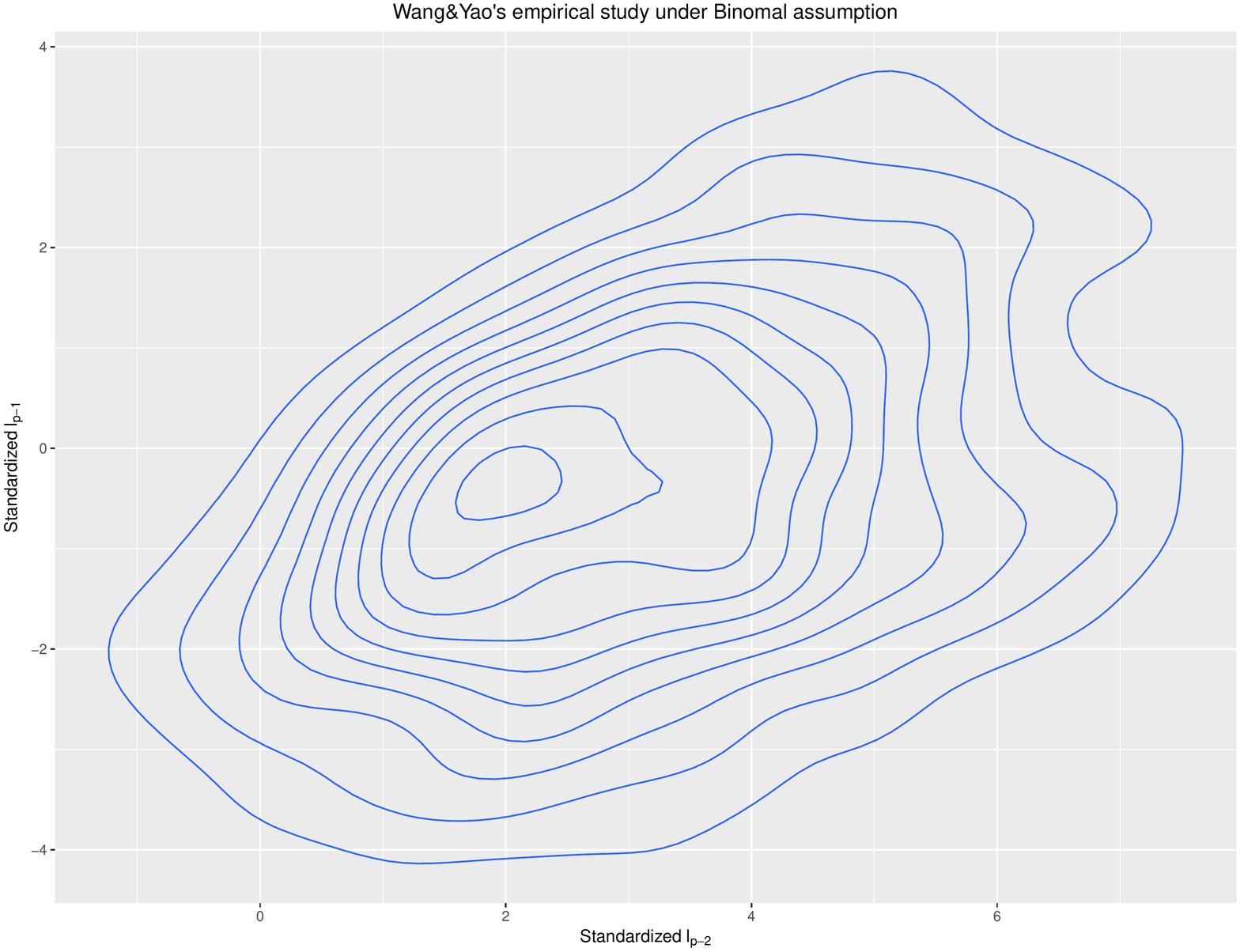}\\
\caption{ {\rm \bf Case~II} under Binomal assumption. Upper panels show that the  Q-Q plots for the proposed $\gamma_1$ and $\gamma_3$, as well as  the empirical densities of $\gamma_1$ and $\gamma_p$  (solid lines) comparing to their Gaussian limits (dashed lines). Middle panels are the corresponding comparison of the empirical density of  standardized $l_{1}$ and $l_p$ in Wang and Yao (2017).  Lower panels show three  contour plots: the first is the proposed empirical joint density function of $(\gamma_{21}, \gamma_{22})$; the second is their corresponding limits; the third is the  empirical joint density function of
standardized $l_{p-2}$ and $l_{p-1}$.
 }\label{fig:5}
\end{center} 
\end{figure}


As seen from the Figures~\ref{fig:4}-\ref{fig:5}, our proposed method performs well for both of the  population assumptions under  {\rm \bf Case~II}, but the method of  \cite{WangYao2017}  provides inaccurate variances  for all the non-Gaussian assumptions under  {\rm \bf Case~II} because the assumption of diagonal block independence is not met.


\section{Conclusion}\label{Con}

In this paper,  a G4MT for a generalized spiked Fisher matrix is proposed. By the relaxing the  matching up to 4th moments  condition to a tail probability in Assumption~$\bB$,  which is necessary for the existence of the largest eigenvalue limit, we show that   the universality of the asymptotic law for the spiked eigenvalues of high-dimensional generalized Fisher matrices. 
To  illustrate the basic idea and procedures of the G4MT,  we apply it to the CLT of normalized spiked eigenvalues
related to high-dimensional generalized Fisher matrices as an example. 
Comparing to the previous work on universality, we simplify the calculations of high-order partial derivatives of an implicit function to the entries of the random matrix, avoid he strong condition $C_0$ of sub-exponential property, and further relax the requirements for the bounded 4th moments and diagonal block independent assumption. Instead,  we only need a more regular and minor conditions (\ref{CondU1}) and (\ref{CondU2}) on the elements of $U_1$ and $V_1$, respectively. On the one hand, our result has much wider applications than \cite{WangYao2017};  on the other hand, the result of 
 \cite{WangYao2017}  shows the necessity of the conditions (\ref{CondU1}) and  (\ref{CondU2}).

\section*{Acknowledgements}

\end{document}